\documentclass[10pt]{amsart}
\usepackage{amsmath,oubraces,amsfonts,amscd,amssymb,graphicx,mathrsfs,eufrak}
\usepackage[all,line,dvips,poly]{xy}
\usepackage[usenames]{color}
\usepackage[bbgreekl]{mathbbol}

\newtheorem{defin}{Definition}[section]
\newtheorem{thm}[defin]{Theorem}
\newtheorem{prop}[defin]{Proposition}
\newtheorem{lemma}[defin]{Lemma}
\newtheorem{cor}[defin]{Corollary}

\newtheorem{example}[defin]{Example}
\newtheorem{remark}[defin]{Remark}

\newcommand{\Mod}[1]{\mathfrak{Mod}#1}
\newcommand{\fgmod}[1]{\mathfrak{mod}#1}

\newcommand{\Qcoh}[1]{\mathfrak{Qcoh} #1}
\newcommand{\coh}[1]{\mathfrak{coh} #1}

\newcommand{\dbcoh}[1]{D^b(\coh{#1})}
\newcommand{\dbqcoh}[1]{D^b(\Qcoh{#1})}

\newcommand{\da}[1]{D(\Mod{#1})}

\newcommand{\dba}[1]{D^b(\Mod{#1})}
\newcommand{\db}[1]{D^b(\fgmod{#1})}

\newcommand{\dqcoh}[1]{D(\Qcoh{#1})}

\newcommand{\stq}[1]{\underline{\Mod{#1}}}

\newcommand{\Sstq}[1]{\textsf{St}(\stq{R})}

\newcommand{\C}[1]{\mathbb{C}^{#1}}
\newcommand{\m}{\mathfrak{m}}

\newcommand{\s}[1]{\mathscr{#1}}
\renewcommand{\c}[1]{\mathcal{#1}}

\renewcommand{\P}{\mathbb{P}^1}

\newcommand{\ve}{\varepsilon}

\newcommand{\thick}[1]{\langle #1\rangle}
\newcommand{\loc}[1]{\langle #1\rangle_{\oplus}}
\newcommand{\T}{\mathfrak{T}}

\newcommand{\X}{\mathfrak{X}}

\newcommand{\cl}[2]{c_{#1 #2}}
\newcommand{\an}[2]{a_{#1 #2}}
\newcommand{\CL}[2]{C_{#1 #2}}
\newcommand{\AN}[2]{A_{#1 #2}}
\newcommand{\CLH}[2]{\hat{C}_{#1 #2}}
\newcommand{\ANH}[2]{\hat{A}_{#1 #2}}
\newcommand{\clt}[2]{\mathtt{c}_{#1 #2}}

\newcommand{\ant}[2]{\mathtt{a}_{#1 #2}}
\newcommand{\CLt}[2]{\mathtt{C}_{#1 #2}}
\newcommand{\ANt}[2]{\mathtt{A}_{#1 #2}}

\newcommand{\w}[1]{\widetilde{#1}}

\renewcommand{\t}[1]{\textnormal{#1}}
\renewcommand{\tt}[1]{\mathtt{#1}}

\begin{document}
\title{\textsc{Reconstruction Algebras of Type $A$}}
\author{Michael Wemyss}
\address{Michael Wemyss\\ Graduate School of Mathematics\\ Nagoya University\\ Chikusa-ku, Nagoya, 464-8602, Japan}
\email{wemyss.m@googlemail.com}
\thanks{2000 Mathematics Subject Classification 16S38, 13C14, 14E15}
\maketitle

\begin{abstract}
We introduce a new class of algebras, called reconstruction
algebras, and present some of their basic properties.  These
non-commutative rings dictate in every way the process of
resolving the Cohen-Macaulay singularities $\mathbb{C}^2/G$ where
$G=\frac{1}{r}(1,a)\leq GL(2,\C{})$.
\end{abstract}
\parindent 20pt
\parskip 0pt

\section{Introduction}
It is not a new idea that non-commutative algebra in many ways
dictates the process of desingularisation in algebraic geometry.  This has been a theme in many recent papers (eg
\cite{vdb_Non_comm_crepant}, \cite{BKR_theorem},
\cite{bridg_local_CY}), however almost all research in this
direction has taken place inside the relatively small sphere of
Gorenstein singularities.  For example, when considering rings of
invariants by small finite subgroups of $GL(n,\C{})$, the Gorenstein
hypothesis forces the subgroups inside $SL(n,\C{})$.

For $G$ a finite subgroup of $SL(2,\mathbb{C})$ it is well known
that the preprojective algebra of the corresponding extended
Dynkin diagram encodes the process of resolving the
Gorenstein Kleinian singularity $\mathbb{C}[x,y]^G$.  From the
viewpoint of this paper, the preprojective algebra should be
treated as an algebra that can be naturally associated to the dual
graph of the minimal commutative resolution, from which we can
gain all information about the process of desingularisation.  Thus
the preprojective algebra is defined with prior knowledge of the
dual graph of the minimal resolution, but since it is Morita
equivalent to the skew group ring we could alternatively use this
purely algebraic ring. The question arises whether there are
similar non-commutative algebras for finite subgroups of
$GL(2,\mathbb{C})$.

The answer is yes \cite{Wemyss_GL2}, and in this paper we prove it for the case of
finite cyclic subgroups $G=\frac{1}{r}(1,a)\leq GL(2,\C{})$ (for
notation see Section 2).

For such a group $G$, we associate to the dual graph of the
minimal commutative resolution (complete with self-intersection
numbers) a non-commutative ring $A_{r,a}$ which we call the
\emph{reconstruction algebra} and prove that $A_{r,a}$ is
isomorphic to the endomorphism ring of the special Cohen-Macaulay
modules in the sense of Wunram ~\cite{Wunram_generalpaper}. This
is important since it shows that for cyclic groups there is a structural correspondence (via the underlying quiver) between the special CM modules and the dual graph complete with self-intersection numbers, thus generalizing McKay's observation for finite subgroups of $SL(2,\C{})$ to finite cyclic subgroups of $GL(2,\C{})$. 

The above is a correspondence purely on the level of the underlying quiver.  However if we also add in the information of the relations we get more: in this paper we prove that the reconstruction
algebra $A_{r,a}$
\begin{itemize}
\item has centre $\C{}[x,y]^{\frac{1}{r}(1,a)}$ and so contains all the information regarding the singularity.  Furthermore it is finitely generated over its centre, so is `tractably' non-commutative (Corollary~\ref{centrefg}).
\item contains enough information to construct the minimal resolution via a
moduli space of finite dimensional representations (Theorem~\ref{modulimain}).
\item contains exactly the same homological information as the minimal resolution through a derived equivalence (Theorem~\ref{derived_equiv}).
\end{itemize}

Although this paper studies cyclic subgroups of $GL(2,\C{})$ and
therefore both the singularities $\C{2}/G$ and their minimal
resolutions are toric, the main ideas in this paper (for example the correspondence between the quiver and the dual graph) are independent of
toric geometry and as such provide the correct framework
for generalisation.

We also remark that in general the reconstruction algebra is
\emph{not} homologically homogenous in the sense of
Brown-Hajarnavis \cite{Brown_Har_HomhomRings}. This should not be
surprising, as there are many other examples of non-commutative
resolutions of sensible non-Gorenstein Cohen-Macaulay
singularities which are not homologically homogeneous
(\cite{Smith_example} and \cite[5.1(2)]{Stafford_VdB}). Non-commutative crepant resolutions have yet to be defined for
Cohen-Macaulay singularities, however when $G\nleq SL(2,\C{})$ the
minimal resolution of $\C{2}/ G$ is not crepant yet is still
important. Hence the rings we produce should certainly be examples
of (non-crepant) non-commutative resolutions, whenever such a
definition is conceived.  The failure of the homologically
homogeneous property suggests we ought to again think hard about
the non-commutative analogue of smoothness.

In fact the reconstruction algebra $A_{r,a}$ should be the
\emph{minimal} non-commutative resolution in some rough sense;
certainly there is the following picture of derived categories:
\[
\xymatrix@C=10pt@R=10pt{{\dbcoh{\w{X}}} & & {\db{\C{}[x,y]\# G}}\\
& {\db{A_{r,a}}}\ar[ul]^{\cong}\ar@{^{(}->}[ur]&}
\]
so we should still perhaps view the skew group ring as a
non-commutative resolution, just not the smallest one.

This paper is organized as follows - in Section 2 we define the
reconstruction algebra associated to a labelled Dynkin diagram of
type $A$ and describe some of its basic structure. In Section 3 we
prove that it is isomorphic to the endomorphism ring of some
Cohen-Macaulay modules. In Section 4 the minimal resolution of the
singularity $\C{2}/ \frac{1}{r}(1,a)$ is obtained via a certain moduli
space of representations of the associated reconstruction algebra
$A_{r,a}$, and in Section 5 we produce a tilting bundle which
gives us our derived equivalence.  In Section 6 we prove that
$A_{r,a}$ is a prime ring and use this to show that the Azumaya
locus of $A_{r,a}$ coincides with the smooth locus of its centre
$\C{}[x,y]^{\frac{1}{r}(1,a)}$.  This then gives a precise value
for the global dimension of $A_{r,a}$, which shows that the
reconstruction algebra need not be homologically homogeneous.

In this paper we work mostly in unbounded derived categories
where arbitrary coproducts exist.  This allows us to use
techniques such as Bousfield localisation and compactly generated
categories to simplify some of the work needed to obtain bounded
derived equivalences, which in turn saves us from having to prove
at the beginning that the reconstruction algebra has finite global
dimension.  Throughout we shall use $\t{D}(\s{A})$ to denote the unbounded
derived category and $\t{D}^b(\s{A})$ to denote the bounded
derived category.  When working with quivers, we shall write $xy$
to mean \textbf{$x$ followed by $y$.}  We work over the ground
field $\mathbb{C}$ but any algebraically closed field of
characteristic zero will suffice.

The moduli results in this paper have been independently
discovered by Alastair Craw \cite{Craw_independant_moduli}.  The benefits of his
approach is that the minimal resolution is produced by using
global arguments (as opposed to my local arguments), however the technique here generalizes to the non-toric case \cite{Wemyss_GL2}.  Also, here the non-commutative ring can be explicitly
written down.  Both approaches have their merits.

This paper formed part of the author's PhD thesis at the University of Bristol, funded by the EPSRC.  Thanks to Aidan Schofield, Ken Brown, Iain Gordon, Alastair Craw and Alastair King.  Thanks also to the anonymous referee whose suggestions greatly improved this paper's readability.

\section{The Reconstruction Algebra of Type $A$}
Consider, for positive integers $\alpha_i\geq 2$, the labelled
Dynkin diagram of type $A_n$:
\[
\xymatrix@C=40pt{ \bullet\ar@{-}[r]^<{-\alpha_n} &
\bullet\ar@{-}[r]^<{-\alpha_{n-1}} &\hdots\ar@{-}[r]
&\bullet\ar@{-}[r]^<{-\alpha_2} & \bullet \ar@{}[r]^<{-\alpha_1}&}
\]
We call the vertex corresponding to $\alpha_i$ the $i^{th}$
vertex. To this picture we associate the double quiver of the
extended Dynkin quiver, with the extended vertex called the
$0^{th}$ vertex:
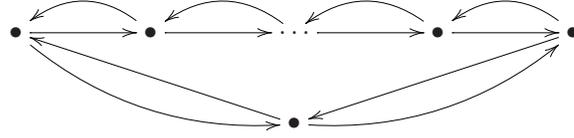
\[
\xymatrix@C=40pt{ \bullet\ar[r]\ar@/_1pc/[1,2] &
\bullet\ar[r]\ar@/_1pc/[l]&\hdots\ar[r]\ar@/_1pc/[l]
&\bullet\ar[r]\ar@/_1pc/[l] & \bullet \ar@/_1pc/[l]\ar[1,-2]&\\
& & \bullet\ar[-1,-2]\ar@/_1pc/[-1,2] & & &}
\]
Name this quiver $Q^\prime$. For the sake of completeness note
that for $n=1$ by $Q^\prime$ we mean
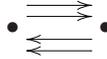
\[
\xymatrix{\bullet\ar@<1ex>[r]\ar@<2ex>[r] &
\bullet\ar@<1ex>[l]\ar@<2ex>[l]}
\]
Now if any $\alpha_i>2$, add an extra $\alpha_i-2$ arrows from the
$i^{th}$ vertex to the $0^{th}$ vertex.   Name this new quiver
$Q$. Notice that when every $\alpha_i=2$, $Q=Q^\prime$ is exactly
the underlying quiver of the preprojective algebra of type
$\w{A}_n$.

We label the arrows in $Q$ as follows:
\[
\begin{array}{rl}
\t{if }n=1 & \t{label the 2 arrows from 0 to 1 in }Q^\prime\,\,\t{by } a_1,a_2\\
 & \t{label the 2 arrows from 1 to 0 in }Q^\prime\,\,\t{by }c_1,c_2\\
 & \t{label the extra arrows due to }\alpha_1\,\,\t{by }k_1,\hdots k_{\alpha_1-2}\\
 & \\
\t{if }n\geq 2 & \t{label the clockwise arrows in }Q^\prime\,\,\t{from
}i\,\,\t{to }i-1\,\,\t{by }\cl{i}{i-1}\,\, \t{(and }\cl{0}{n}\t{)} \\
 & \t{label the anticlockwise arrows in }Q^\prime\,\,\t{from
}i\,\,\t{to }i+1\,\,\t{by }\an{i}{i+1}\,\,\t{(and }\an{n}{0}\t{)}\\
 & \t{label the extra arrows by }k_1,\hdots,k_{\sum (\alpha_i-2)}\,\,\t{anticlockwise}
\end{array}
\]
Note for example that $\cl{1}{2}$ should be read `clockwise from 1 to 2'.
It is also convenient to write $\AN{i}{j}$ for the composition of
anticlockwise paths $a$ from vertex $i$ to vertex $j$, and
similarly $\CL{i}{j}$ as the composition of clockwise paths, where
by $\CL{i}{i}$ (resp. $\AN{i}{i}$) we mean not the empty path at
vertex $i$ but the path from $i$ to $i$ round each of the
clockwise (resp. anticlockwise) arrows precisely once.  For
convenience we also denote $\cl{1}{0}:=k_0$ and $\an{n}{0}:=k_{1+\sum
(\alpha_i-2)}$.

\begin{example}\label{4,2}\t{ For $[\alpha_1,\alpha_2]=[4,2]$ the
quiver $Q$ is}
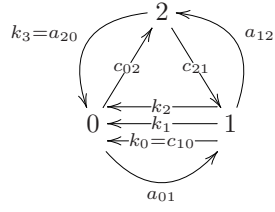
\[
\xymatrix@R=30pt@C=15pt{ & 2\ar|{\cl{2}{1}}[dr]\ar@/_1.5pc/_{k_3=\an{2}{0}}[dl] \\
0\ar[ur]|{\cl{0}{2}}\ar@<-1ex>@/_1.5pc/_{\an{0}{1}}[rr] & &
1\ar@<1.5ex>|{k_0=\cl{1}{0}}[ll]\ar|{k_1}[ll]\ar|{k_2}@<-1.5ex>[ll]\ar@/_1.5pc/_{\an{1}{2}}[ul]}
\]
\end{example}
\begin{example}\label{4,3,4} \t{For $[\alpha_1,\alpha_2,\alpha_3]=[4,3,4]$ the quiver
$Q$ is}
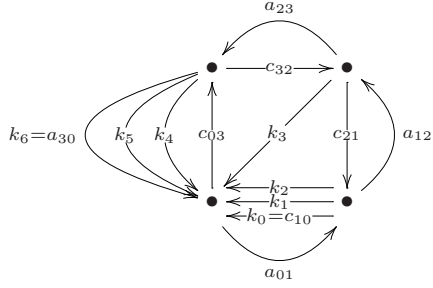
\[
\xymatrix@R=40pt@C=40pt{\bullet\ar|{\cl{3}{2}}[r]\ar@/_1.5pc/|{k_4}[d]\ar@/_2.75pc/|{k_5}[d]\ar@/_4pc/_{k_6=\an{3}{0}}[d]
&
\bullet\ar|{\cl{2}{1}}[d]\ar@/_1.5pc/_{\an{2}{3}}[l]\ar@<-0.25ex>|{k_3}[dl] \\
\bullet\ar[u]|{\cl{0}{3}}\ar@<-1ex>@/_1.5pc/_{\an{0}{1}}[r]  &
\bullet\ar@<1.25ex>|{k_0=\cl{1}{0}}[l]\ar|{k_1}[l]\ar|{k_2}@<-1.25ex>[l]\ar@/_1.5pc/_{\an{1}{2}}[u]}
\]
\end{example}
Denote by $l_{r}$ the number of the vertex associated to the tail
of the arrow $k_r$ and denote $u_i:=\t{max}\{ j: l_{j}=i \}$ and
$v_i:=\t{min}\{ j: l_{j}=i  \}$.  Because we have defined
$k_0:=\cl{1}{0}$ and $k_{1+\sum (\alpha_i-2)}:=\an{n}{0}$ it is
always true that $v_1=0$ and $u_n=1+\sum( \alpha_i-2)$.  For
$2\leq i\leq n$ write $V_i:=\t{max}\{ j: l_{j}< i \}$ and set
$V_1:=0$.  In Example~\ref{4,3,4} above $u_1=2$, $v_3=4$, $V_{l_5}=V_3=3$
and $V_{l_2}=V_1=0$.

\begin{defin}\label{reconstruct a_1,a_2,...,a_n}
For labels $[\alpha_1,\hdots,\alpha_n]$ with each $\alpha_i\geq
2$, define the reconstruction algebra of type $A$ as the path
algebra of
the quiver $Q$ subject to the following relations:\\

$
\begin{array}{rl}
\mbox{if } n=1 & c_2a_1=c_1a_2\mbox{ and }a_1c_2=a_2c_1\\
 & k_1a_1=c_2a_2\mbox{ and } a_1k_1=a_2c_2 \\
 & k_ta_1=k_{t-1}a_2\mbox{ and } a_1k_{t}=a_2k_{t-1}\,\,\forall\,\,
 2\leq t\leq \alpha_1-2.\end{array}$\\

$\begin{array}{cccl} \mbox{if }n\geq 2 & \mbox{Step 1:} & \mbox{If }\alpha_1=2  & \cl{1}{0}\an{0}{1}=\an{1}{2}\cl{2}{1}\\
& & \mbox{If }\alpha_1>2 & k_s\an{0}{1}=k_{s+1}\CL{0}{1},
\an{0}{1}k_s=\CL{0}{1}k_{s+1}\,\,\forall\,\,0\leq s< u_1\\ & & &
k_{u_1}\an{0}{1}=\an{1}{2}\cl{2}{1}.\\ &&&\vdots \\
 & \mbox{Step t:} & \mbox{If }\alpha_t=2  & \cl{t}{t-1}\an{t-1}{t}=\an{t}{t+1}\cl{t+1}{t}\\
& & \mbox{If }\alpha_t>2 &
\cl{t}{t-1}\an{t-1}{t}=k_{v_t}\CL{0}{t}, \CL{0}{t}k_{v_t}=\AN{0}{l_{{V_t}}}k_{V_t} \\
& & & k_{s}\AN{0}{t}=k_{s+1}\CL{0}{t},
\AN{0}{t}k_s=\CL{0}{t}k_{s+1}\,\,\forall\,\,v_t\leq s< u_t\\
& & & k_{u_t}\AN{0}{t}=\an{t}{t+1}\cl{t+1}{t}\\ &&&\vdots \\
 & \mbox{Step n:} & \mbox{If }\alpha_n=2  & \cl{n}{n-1}\an{n-1}{n}=\an{n}{0}\cl{0}{n}, \cl{0}{n}\an{n}{0}=\AN{0}{l_{{V_n}}}k_{V_n}\\
& & \mbox{If }\alpha_n>2 &
\cl{n}{n-1}\an{n-1}{n}=k_{v_n}\cl{0}{n},\cl{0}{n}k_{v_n}=\AN{0}{l_{{V_n}}}k_{V_n}\\
& & & k_{s}\AN{0}{n}=k_{s+1}\cl{0}{n},
\AN{0}{n}k_s=\cl{0}{n}k_{s+1}\,\,\forall\,\,v_n\leq s< u_n
\end{array}$
\end{defin}
\begin{example}\label{4,2 relations}
\t{The reconstruction algebra of type $A$ associated to $[4,2]$
is}
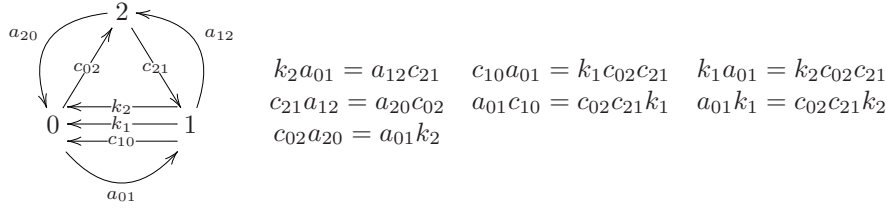
\[
\begin{array}{c}
\xymatrix@R=30pt@C=15pt{ & 2\ar|{\cl{2}{1}}[dr]\ar@/_1.5pc/_{\an{2}{0}}[dl] \\
0\ar[ur]|{\cl{0}{2}}\ar@<-1ex>@/_1.5pc/_{\an{0}{1}}[rr] & &
1\ar@<1.5ex>|{\cl{1}{0}}[ll]\ar|{k_1}[ll]\ar|{k_2}@<-1.5ex>[ll]\ar@/_1.5pc/_{\an{1}{2}}[ul]}\end{array}
\begin{array}{ccc}
k_2\an{0}{1}=\an{1}{2}\cl{2}{1}  &\cl{1}{0}\an{0}{1}=k_1\cl{0}{2}\cl{2}{1}  & k_1\an{0}{1}=k_2\cl{0}{2}\cl{2}{1}\\
\cl{2}{1}\an{1}{2}=\an{2}{0}\cl{0}{2} & \an{0}{1}\cl{1}{0}=\cl{0}{2}\cl{2}{1}k_1 & \an{0}{1}k_1=\cl{0}{2}\cl{2}{1}k_2\\
\cl{0}{2}\an{2}{0}=\an{0}{1}k_2 & &
\end{array}
\]
\end{example}
\begin{example}\label{4,3,4_relations}
\t{The reconstruction algebra of type $A$ associated to $[4,3,4]$
is the path algebra of the quiver in Example~\ref{4,3,4} subject
to the relations}
\begin{align*}
\cl{1}{0}\an{0}{1}&=k_1\cl{0}{3}\cl{3}{2}\cl{2}{1} &{} \an{0}{1}\cl{1}{0}&=\cl{0}{3}\cl{3}{2}\cl{2}{1}k_1 &{} \\
k_1\an{0}{1}&=k_2\cl{0}{3}\cl{3}{2}\cl{2}{1} &{}
\an{0}{1}k_1&=\cl{0}{3}\cl{3}{2}\cl{2}{1}k_2 &{}\\
k_2\an{0}{1}&=\an{1}{2}\cl{2}{1} &{}
\cl{2}{1}\an{1}{2}&=k_3\cl{0}{3}\cl{3}{2}&{}
\cl{0}{3}\cl{3}{2}k_3&=\an{0}{1}k_2\\
k_3\an{0}{1}\an{1}{2}&=\an{2}{3}\cl{3}{2} &{}
\cl{3}{2}\an{2}{3}&=k_4\cl{0}{3}&{}
\cl{0}{3}k_4&=\an{0}{1}\an{1}{2}k_3\\
k_4\an{0}{1}\an{1}{2}\an{2}{3}&=k_5\cl{0}{3} &{} \an{0}{1}\an{1}{2}\an{2}{3}k_4&=\cl{0}{3}k_5 &{} \\
k_5\an{0}{1}\an{1}{2}\an{2}{3}&=\an{3}{0}\cl{0}{3} &{}
\an{0}{1}\an{1}{2}\an{2}{3}k_5&=\cl{0}{3}\an{3}{0} &{}
\end{align*}
\end{example}

\begin{defin}
For $r,a\in\mathbb{N}$ with $\t{hcf}(r,a)=1$ and $r>a$ define the group
$G=\frac{1}{r}(1,a)$ by
\[
G= \left\langle \zeta:=\left(\begin{array}{cc} \ve & 0\\ 0& \ve^a
\end{array}\right)\right\rangle \leq GL(2,\C{}),
\]
where $\ve$ is a primitive $r^{th}$ root of unity.
\end{defin}

Now consider the Jung-Hirzebruch continued fraction expansion of
$\frac{r}{a}$, namely
\[
\frac{r}{a}=\alpha_1-\frac{1}{\alpha_2 - \frac{1}{\alpha_3 -
\frac{1}{(...)}}} :=[\alpha_1,\hdots,\alpha_n]
\]
with each $\alpha_i\geq 2$.  The labelled Dynkin diagram of type
$A$ associated to this data is precisely the dual graph of the
minimal resolution of $\C{2}/\frac{1}{r}(1,a)$ \cite[Satz8]{Riemenschneider_invarianten}.

\begin{defin}\label{reconstruct A_r,a}
Define the reconstruction algebra $A_{r,a}$ associated to the
group $G=\frac{1}{r}(1,a)$ to be the reconstruction algebra of
type $A$ associated to the data of the Jung-Hirzebruch continued
fraction expansion of $\frac{r}{a}$.
\end{defin}
Note for the group $\frac{1}{r}(1,r-1)$, the reconstruction
algebra $A_{r,r-1}$ is the reconstruction algebra of type $A$ for
the data $[\underbrace{2,\hdots,2}_{r-1}]$.  Since $V_n=0$,
$k_{V_n}=\cl{1}{0}$ and $l_{V_n}=1$ this is precisely the
preprojective algebra of type $\w{A}_{r-1}$.

\begin{example}\t{Since $\frac{7}{2}=[4,2]$ the
reconstruction algebra $A_{7,2}$ associated to the group
$\frac{1}{7}(1,2)$ is precisely the algebra in Example~\ref{4,2
relations}}.
\end{example}

\begin{example}\t{After noticing that $\frac{40}{11}=[4,3,4]$ we see that the
reconstruction algebra $A_{40,11}$ associated to the group
$\frac{1}{40}(1,11)$ is precisely the algebra in
Example~\ref{4,3,4_relations}}.
\end{example}
The following lemma is important later for certain duality
arguments; geometrically it says that the reconstruction algebra
is independent of the direction we view the dual graph of the
minimal resolution:

\begin{lemma}\label{read_backwards}
The reconstruction algebra of type $A$ associated to the data
$[\alpha_1,\hdots,\alpha_n]$ is the same as that associated to the
data $[\alpha_n,\hdots,\alpha_1]$.
\end{lemma}
\begin{proof}
If $n=1$ there is nothing to prove, so assume $n\geq 2$.  To avoid
confusion write everything to do with the reconstruction algebra
associated to $[\alpha_n,\hdots,\alpha_1]$ in typeface fonts, e.g.\
$\ant{0}{n}$, $\ANt{0}{3}$, $\clt{1}{2}$ $\CLt{0}{n-1}$,
$\tt{k}_1$, $\tt{u}_n$, etc.  Flip the quiver vertex numbers by the
operation $^\prime$ which takes $0$ to itself (ie $0^\prime=0$),
and reflects the other vertices in the natural line of symmetry
(i.e.\ $1^\prime =n$, $n^\prime=1$ etc).  Then the explicit
isomorphism between the reconstruction algebras is given by
\begin{align*}
\cl{i}{j}&\mapsto \ant{i^\prime}{j^\prime}\\
\an{i}{j}&\mapsto \clt{i^\prime}{j^\prime}\\
k_i&\mapsto \tt{k}_{n-i}
\end{align*}
Under this map $\AN{i}{j}\mapsto \CLt{i^\prime}{j^{\prime}}$ and
$\CL{i}{j}\mapsto \ANt{i^\prime}{j^{\prime}}$, and furthermore the
relations for the reconstruction algebra associated to
$[\alpha_1,\hdots,\alpha_n]$ read backwards are precisely the
relations for the reconstruction algebra associated to
$[\alpha_n,\hdots,\alpha_1]$ read forwards.
\end{proof}

Now $A_{r,a}$ is supposed to encode all information about the
singularity $\C{}[x,y]^{\frac{1}{r}(1,a)}$ as well as the
resolution, so since $\C{}[x,y]^{\frac{1}{r}(1,a)}$ is determined
by the continued fraction expansion of $\frac{r}{r-a}$
(by \cite[Satz1]{Riemenschneider_invarianten}; see Lemma~\ref{generators} below), it should be possible to read this directly from the quiver.  Indeed this is true and to do
it we must introduce some more notation.

Define $\sigma_1=1$ and inductively $\sigma_s$ ($s\geq 1$) to be
the smallest vertex $t$ with $t>\sigma_{s-1}$ and $\alpha_t>2$ (if
it exists), otherwise $\sigma_s=n$. Stop this process when we reach
$n$. Thus we have
\[
1= \sigma_1<\hdots<\sigma_z= n.
\]
Note that if all $\alpha_t=2$, this degenerates into
$1=\sigma_1<\sigma_2=n$.
\begin{lemma}\label{Reimen_duality}
For the group $\frac{1}{r}(1,a)$ with notation as above
\[
\frac{r}{r-a}=[\underbrace{2,\hdots,2}_{u_{\sigma_1}-v_{\sigma_1}},(\sigma_2-\sigma_1+2),\underbrace{2,\hdots,2}_{u_{\sigma_2}-v_{\sigma_2}},(\sigma_3-\sigma_2+2),\underbrace{2,\hdots,2}_{u_{\sigma_3}-v_{\sigma_3}},\hdots,(\sigma_z-\sigma_{z-1}+2),\underbrace{2,\hdots,2}_{u_{\sigma_z}-v_{\sigma_z}}]
\]
\end{lemma}
\begin{proof}
This is just a reformulation of Riemenschneider duality, using the reconstruction algebra to give a slightly different interpretation of the Riemenschneider point diagram (see \cite[p223]{Reimen_deform}).  See also \cite[1.2]{Kidoh}.
\end{proof}
\begin{example}
\t{By looking at the shape of $A_{40,11}$ in
Example~\ref{4,3,4}, by the above}
\[
\frac{40}{40-11}=[2,2,3,3,2,2].
\]
\end{example}

Thus the shape of the reconstruction algebra $A_{r,a}$ determines
the continued fraction expansion of $\frac{r}{r-a}$, which in turn
determines the singularity $\C{}[x,y]^{\frac{1}{r}(1,a)}$ (see Lemma~\ref{generators}).  We
will prove in Corollary~\ref{centrefg} that $Z(A_{r,a})=\C{}[x,y]^{\frac{1}{r}(1,a)}$.

\section{Special Cohen-Macaulay Modules}
The reconstruction algebra is, by definition, constructed with
prior knowledge of the minimal resolution.  The aim of this
section is to show that we could have defined it in a purely
algebraic way by summing certain CM modules and
looking at their endomorphism ring.  More precisely, in this
section we shall show that the reconstruction algebra is
isomorphic as a ring to the endomorphism ring of the sum of the
special CM modules. In the
process, we shall see that the relations for the reconstruction
algebra arise naturally through a notion which we call a web
of paths.

Keeping the notation from the last section, consider the group
$G=\frac{1}{r}(1,a)=\langle \zeta\rangle$.
\begin{defin}
For $0\leq t\leq r-1$ define
\[
S_t=\{ f\in\C{}[x,y]: \zeta f=\varepsilon^t f  \}.
\]
\end{defin}
These are precisely the non-isomorphic indecomposable maximal
CM modules \cite[10.10]{Yoshino1} over the
CM singularity $X=\t{Spec }\C{}[x,y]^G$.  Of these,
only some are important:
\begin{defin}\cite{Wunram_generalpaper}
The module $S_t$ is said to be special if $S_t\otimes \omega_X /
\t{torsion}$ is CM, where $\omega_X$ is the canonical
module of $X=\t{Spec }\C{}[x,y]^G$.
\end{defin}
Note that the ring $S_0$ is always special.  There are in fact many equivalent characterisations of the special CM modules (see for example \cite[Thm
5]{Riemenschneider_specials}), some which refer to the minimal
resolution and some that do not.  For cyclic groups the
combinatorics governing which CM modules are special is well
understood.
\begin{defin}
For $\frac{r}{a}=[\alpha_1,\hdots,\alpha_n]$ define the $i$-series
and $j$-series as follows:
\[
\begin{array}{ccl}
i_0=r & i_1=a & i_{t}=\alpha_{t-1}i_{t-1}-i_{t-2}\,\mbox{ for }\,2\leq
t\leq n+1,\\ j_0=0 & j_1=1 &
j_{t}=\alpha_{t-1}j_{t-1}-j_{t-2}\,\mbox{ for }\,2\leq t\leq n+1.
\end{array}
\]
\end{defin}
It's easy to see that
\[
\begin{array}{ccccccccccc}
i_0=r &>& i_1=a &>& i_2 &>& \hdots &>& i_n=1 &>& i_{n+1}=0, \\
j_0=0 &<& j_1=1 &<& j_2=\alpha_1&<& \hdots &<& j_n &<& j_{n+1}=r.
\end{array}
\]
It is the $i$-series which gives an easy combinatorial
characterisation of the specials:
\begin{thm}\cite{Wunram_cyclicBook}\label{Wunram_special_characteristation}
For $G=\frac{1}{r}(1,a)$ with
$\frac{r}{a}=[\alpha_1,\hdots,\alpha_n]$, the special
CM modules are precisely those $S_{i_p}$ for $0\leq
p\leq n$. Furthermore if $1\leq p\leq n$, then $S_{i_p}$ is
minimally generated by $x^{i_p}$ and $y^{j_p}$.
\end{thm}
For $\frac{r}{a}=[\alpha_1,\hdots, \alpha_n]$ we now sum the
specials and look at the endomorphism ring.  Since $\t{Hom}(S_{i_q},S_{i_p})\cong S_{i_p-i_q}$,
certainly there are the following maps between the specials:
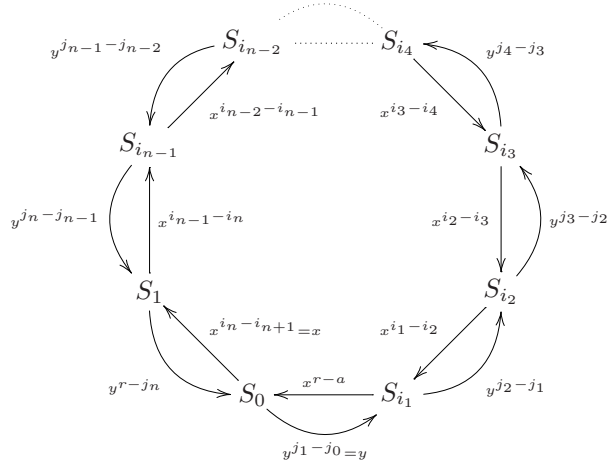
\[
\def\alphanum{\ifcase\xypolynode\or S_{i_3}\or S_{i_4}\or S_{i_{n-2}}\or S_{i_{n-1}}\or S_1
\or S_{0}\or S_{i_1}\or S_{i_2}  \fi}
\xy \xygraph{!{/r6pc/:} 
[] !P8"A"{~>{} ~*{\alphanum}}
}
\ar_{{}_{x^{i_{3}-i_{4}}}} "A2";"A1"
\ar@{.} "A3";"A2"
\ar_{{}_{x^{i_{n-2}-i_{n-1}}}} "A4";"A3"
\ar_{{}_{x^{i_{n-1}-i_{n}}}} "A5";"A4"
\ar_{{}_{x^{i_{n}-i_{n+1}}=x}}  "A6";"A5"
\ar_{{}_{x^{i_{2}-i_{3}}}} "A1";"A8"
\ar_{{}_{x^{i_{1}-i_{2}}}} "A8";"A7"
\ar_{{}_{x^{r-a}}} "A7";"A6"
\ar@/_1.25pc/_{{}_{y^{j_4-j_3}}} "A1";"A2"
\ar@/_1.25pc/@{.} "A2";"A3"
\ar@/_1.25pc/_{{}_{y^{j_{n-1}-j_{n-2}}}} "A3";"A4"
\ar@/_1.25pc/_{{}_{y^{j_n-j_{n-1}}}} "A4";"A5"
\ar@/_1.25pc/_{{}_{y^{r-j_n}}} "A5";"A6"
\ar@/_1.25pc/_{{}_{y^{j_1-j_0}=y}} "A6";"A7"
\ar@/_1.25pc/_{{}_{y^{j_2-j_1}}} "A7";"A8"
\ar@/_1.25pc/_{{}_{y^{j_3-j_2}}} "A8";"A1"
\endxy
\]
In general there will be more.  If for any $1\leq p\leq n$ it is
true that $\alpha_p>2$, then for each $t$ such that $1\leq t\leq
\alpha_p-2$, add an extra map from $S_{i_p}\rightarrow S_0$
labelled $x^{i_{p-1}-(t+1)i_{p}}y^{tj_p-j_{p-1}}$.  Call the
diagram complete with these extra arrows $D$.  Define the natural map
$\phi:\C{}Q\rightarrow \t{End}(\oplus_{p=1}^{n+1}S_{i_p})$ by
\[
\begin{array}{rcl}
\cl{0}{n} &\mapsto & x^{i_n-i_{n+1}}=x \\
\cl{t}{t-1} &\mapsto & x^{i_{t-1}-i_{t}}\\
\an{n}{0} &\mapsto & y^{j_{n+1}-j_{n}}=y^{r-j_n} \\
\an{t}{t+1} &\mapsto & y^{j_{t+1}-j_t}\\
k_s &\mapsto &
x^{i_{l_s-1}-((s-V_{l_s})+1)i_{l_s}}y^{(s-V_{l_s})j_{l_s}-j_{l_s-1}}
\end{array}
\]
where recall $V_1:=0$.  The remainder of this section is devoted to proving that $\phi$ is surjective, with kernel generated by the reconstruction algebra relations.

We first show surjectivity by proving that
$D$ contains all the generators of the maps between the specials, in that
every other map between the specials is a finite sum of
compositions of those in $D$. 

To do this we argue that for any $0\leq p\leq n$ we can see every
$f\in\t{Hom}(S_{i_p},S_{i_p})\cong \C{}[x,y]^G$ as a finite sum of
compositions of arrows in $D$ forming a cycle at
vertex $p$.  We then argue that given any two specials
$S_{i_p},S_{i_q}$ we can see every
$f\in\t{Hom}(S_{i_p},S_{i_q})\cong S_{i_q-i_p}$ as a finite sum of
compositions of arrows in $D$ from vertex $p$ to
vertex $q$.  

We begin by putting the generators of the ring $\C{}[x,y]^{\frac{1}{r}(1,a)}$ into a form suitable for our needs:
\begin{lemma}\label{generators}
$\C{}[x,y]^{\frac{1}{r}(1,a)}$ is generated as a ring by the following invariants:
\[
\begin{array}{rc}
& x^r\\
& x^{r-a}y\\
\t{if }\alpha_1>2 & \left\{ \begin{array}{c} x^{i_0-2i_1}y^{2j_1-j_0}\\ \vdots \\
x^{i_0-(\alpha_1-1)i_1}y^{(\alpha_1-1)j_1-j_0}
\end{array}\right. \\
& \vdots 
\end{array}
\]
\[
\begin{array}{rc}
\t{if }\alpha_n>2 & \left\{ \begin{array}{c} x^{i_{n-1}-2i_{n}}y^{2j_{n}-j_{n-1}}\\
\vdots \\
x^{i_{n-1}-(\alpha_n-1)i_{n}}y^{(\alpha_n-1)j_{n}-j_{n-1}}
\end{array}\right.\\
& y^r
\end{array}
\]
where $i$ and $j$ are the $i$ and $j$-series for the continued fraction expansion of $\frac{r}{a}$.
\end{lemma}
\begin{proof}
Denote by $\mathbb{i}$ and $\mathbb{j}$ the $i$- and $j$-series for the continued fraction expansion of $\frac{r}{r-a}$.  Then it is well known that the collection $x^{\mathbb{i}_t}y^{\mathbb{j}_t}$ ($0\leq t\leq 2+\sum_{p=1}^{n}(\alpha_p-2)$) generate the invariant ring \cite[Satz1]{Riemenschneider_invarianten}.  We must put $\mathbb{i}$ and $\mathbb{j}$ in terms of the $i$ and $j$.  By definition $x^{\mathbb{i}_0}y^{\mathbb{j}_0}=x^r$ and $x^{\mathbb{i}_1}y^{\mathbb{j}_1}=x^{r-a}y=x^{i_0-i_1}y^{j_1-j_0}$, and so the first two terms in the above list are correct.  

\emph{Case 1: $\alpha_1>2$.} Then $u_{\sigma_1}-v_{\sigma_1}=\alpha_1-2>0$, and so by Lemma~\ref{Reimen_duality} $\mathbb{i}_2=2\mathbb{i}_1-\mathbb{i}_0=i_0-2i_1$ and  $\mathbb{j}_2=2\mathbb{j}_1-\mathbb{j}_0=2j_1-j_0$, verifying the next in the list.  Since the first $\alpha_1-2$ entries in the continued fraction expansion of $\frac{r}{r-a}$ are $2$'s, the first `$\alpha_1>2$ block' now follows easily. 

\emph{Case 2: $\alpha_1=2$.} Then the `$\alpha_{2}>2$ block' is empty.

In either case the last term we know to be correct is $x^{i_0-(\alpha_1-1)i_1}y^{(\alpha_1-1)j_1-j_0}$, and the next block to check is the `$\alpha_{\sigma_2}>2$ block'.  To show that in this block the first term is correct, we must prove (by definition of $\mathbb{i}$ and $\mathbb{j}$, using Lemma~\ref{Reimen_duality} to tell us that the next term in the continued fraction expansion of $\frac{r}{r-a}$ is $\sigma_2-\sigma_1+2$) that 
\[
(\sigma_2-\sigma_1+2)(i_0-(\alpha_1-1)i_1)-(i_0-(\alpha_1-2)i_1)=i_{\sigma_2-1}-2i_{\sigma_2}
\]
and 
\[
(\sigma_2-\sigma_1+2)((\alpha_1-1)j_1-j_0)-((\alpha_1-2)j_1-j_0)=2j_{\sigma_2}-j_{\sigma_2-1}.
\]
We show the first; the second is very similar.  By grouping terms, the left hand side ($LHS$) equals $\sigma_2i_0-(\sigma_2\alpha_1-\sigma_2+1)i_1$.  But by the definition of $\sigma_2$ it is easy to show that  $i_{\sigma_2}=(\sigma_2-1)i_2-(\sigma_2-2)i_1$, and so on substituting in $i_2=\alpha_1i_1-i_0$ and grouping terms we obtain $-i_{\sigma_2}=LHS+(\alpha_1-1)i_1-i_0$.  But $(\alpha_1-1)i_1-i_0=i_2-i_1=\hdots=i_{\sigma_2}-i_{\sigma_2-1}$ gives $i_{\sigma_2-1}-2i_{\sigma_2}=LHS$, as required.

From the above we know that the first term in the `$\alpha_{\sigma_2}>2$ block' is correct, and by Lemma~\ref{Reimen_duality} that the next $u_{\sigma_2}-v_{\sigma_2}$ terms in the continued fraction expansion of $\frac{r}{r-a}$ are all $2$'s.  Thus the proof continues exactly as in Case 1 above, from which the induction step is clear.
\end{proof}

We now illustrate the correspondence between the generators of the invariant ring $\C{}[x,y]^{\frac{1}{r}(1,a)}$ and the cycles in $D$ in an example.
\begin{example}
\t{Consider the group $\frac{1}{73}(1,27)$.  In this case the diagram $D$ is}
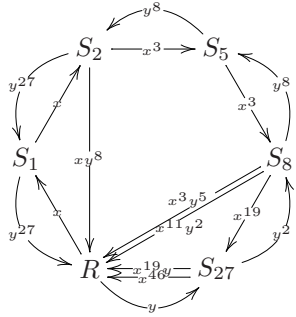
\[
\def\alphanum{\ifcase\xypolynode\or S_8\or S_5\or S_2\or S_1\or R
\or S_{27}\fi}
\xy \xygraph{!{/r4pc/:} 
[] !P6"A"{~>{} ~*{\alphanum}}
}
\ar|{{}_{x^3}} "A2";"A1"
\ar|{{}_{x^3}} "A3";"A2"
\ar|{{}_{x}} "A4";"A3"
\ar|{{}_{x}} "A5";"A4"
\ar|{{}_{x^{19}y}}  "A6";"A5"
\ar|{{}_{x^{19}}} "A1";"A6"
\ar@/_1.25pc/|{{}_{y^8}} "A1";"A2"
\ar@/_1.25pc/|{{}_{y^8}} "A2";"A3"
\ar@/_1.25pc/|{{}_{y^{27}}} "A3";"A4"
\ar@/_1.25pc/|{{}_{y^{27}}} "A4";"A5"
\ar@/_1.25pc/|{{}_{y}} "A5";"A6"
\ar@/_1.25pc/|{{}_{y^2}} "A6";"A1"
\ar|{{}_{xy^8}}"A3";"A5"
\ar|(0.4){{}_{x^3y^5}\quad}@<-0.25ex> "A1";"A5"
\ar|(0.6){\quad {}_{x^{11}y^2}}@<0.25ex> "A1";"A5"
\ar|{{}_{x^{46}}}@<0.75ex>"A6";"A5"
\endxy
\]
\t{and we view the generators of the invariant ring $\C{}[x,y]^{\frac{1}{73}(1,27)}$ as follows:}
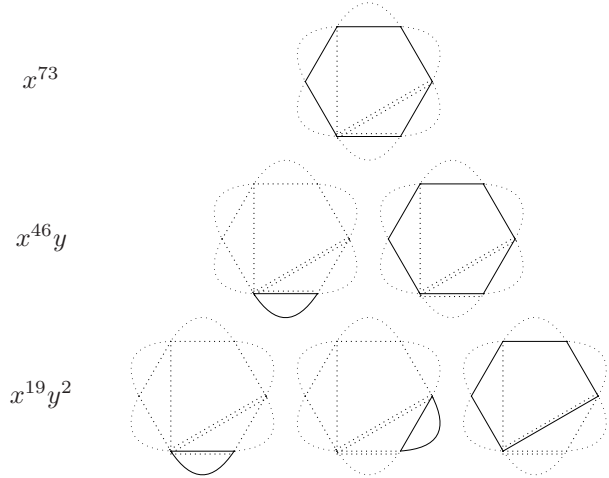
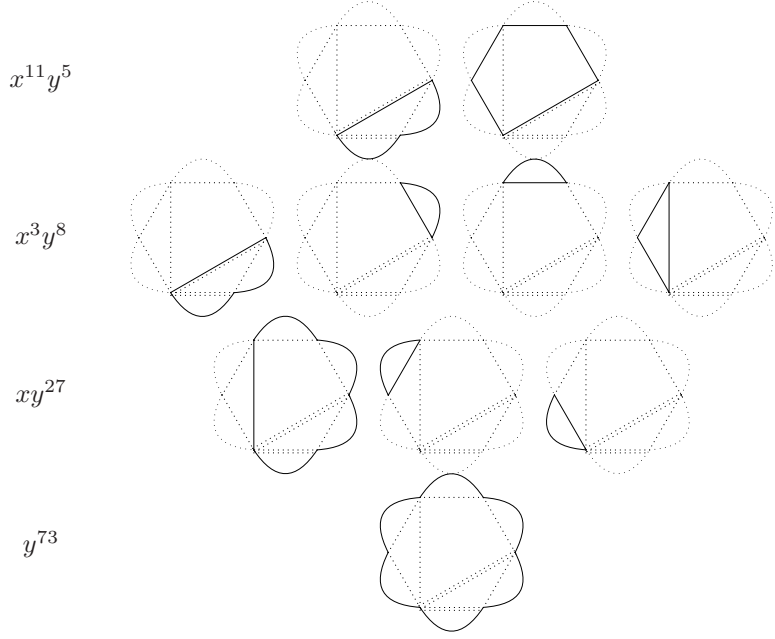
\[
\begin{array}{ccc}
x^{73}&&
\begin{array}{c}
\xy \xygraph{!{/r2pc/:} 
[] !P6"A"{~>{} }
}
\ar@{-} "A2";"A1"
\ar@{-} "A3";"A2"
\ar@{-} "A4";"A3"
\ar@{-} "A5";"A4"
\ar@{-}  "A6";"A5"
\ar@{-} "A6";"A1"
\ar@{.}@/_0.75pc/ "A1";"A2"
\ar@{.}@/_0.75pc/ "A2";"A3"
\ar@{.}@/_0.75pc/ "A3";"A4"
\ar@{.}@/_0.75pc/ "A4";"A5"
\ar@{.}@/_0.75pc/ "A5";"A6"
\ar@{.}@/_0.75pc/ "A6";"A1"
\ar@{.}"A3";"A5"
\ar@{.}@<-0.15ex> "A1";"A5"
\ar@{.}@<0.15ex> "A1";"A5"
\ar@{.}@<-0.25ex>"A6";"A5"
\endxy
\end{array}\\
x^{46}y&&
\begin{array}{cc}
\xy \xygraph{!{/r2pc/:} 
[] !P6"A"{~>{} }
}
\ar@{.} "A2";"A1"
\ar@{.} "A3";"A2"
\ar@{.} "A4";"A3"
\ar@{.} "A5";"A4"
\ar@{-}  "A6";"A5"
\ar@{.} "A6";"A1"
\ar@{.}@/_0.75pc/ "A1";"A2"
\ar@{.}@/_0.75pc/ "A2";"A3"
\ar@{.}@/_0.75pc/ "A3";"A4"
\ar@{.}@/_0.75pc/ "A4";"A5"
\ar@{-}@/_0.75pc/ "A5";"A6"
\ar@{.}@/_0.75pc/ "A6";"A1"
\ar@{.}"A3";"A5"
\ar@{.}@<-0.15ex> "A1";"A5"
\ar@{.}@<0.15ex> "A1";"A5"
\ar@{.}@<-0.25ex>"A6";"A5"
\endxy &
\xy \xygraph{!{/r2pc/:} 
[] !P6"A"{~>{} }
}
\ar@{-} "A2";"A1"
\ar@{-} "A3";"A2"
\ar@{-} "A4";"A3"
\ar@{-} "A5";"A4"
\ar@{-}  "A6";"A5"
\ar@{-} "A6";"A1"
\ar@{.}@/_0.75pc/ "A1";"A2"
\ar@{.}@/_0.75pc/ "A2";"A3"
\ar@{.}@/_0.75pc/ "A3";"A4"
\ar@{.}@/_0.75pc/ "A4";"A5"
\ar@{.}@/_0.75pc/ "A5";"A6"
\ar@{.}@/_0.75pc/ "A6";"A1"
\ar@{.}"A3";"A5"
\ar@{.}@<-0.15ex> "A1";"A5"
\ar@{.}@<0.15ex> "A1";"A5"
\ar@{.}@<0.25ex>"A6";"A5"
\endxy
\end{array}\\
x^{19}y^2&&
\begin{array}{ccc}
\xy \xygraph{!{/r2pc/:} 
[] !P6"A"{~>{} }
}
\ar@{.} "A2";"A1"
\ar@{.} "A3";"A2"
\ar@{.} "A4";"A3"
\ar@{.} "A5";"A4"
\ar@{-}  "A6";"A5"
\ar@{.} "A6";"A1"
\ar@{.}@/_0.75pc/ "A1";"A2"
\ar@{.}@/_0.75pc/ "A2";"A3"
\ar@{.}@/_0.75pc/ "A3";"A4"
\ar@{.}@/_0.75pc/ "A4";"A5"
\ar@{-}@/_0.75pc/ "A5";"A6"
\ar@{.}@/_0.75pc/ "A6";"A1"
\ar@{.}"A3";"A5"
\ar@{.}@<-0.15ex> "A1";"A5"
\ar@{.}@<0.15ex> "A1";"A5"
\ar@{.}@<0.25ex>"A6";"A5"
\endxy &
\xy \xygraph{!{/r2pc/:} 
[] !P6"A"{~>{} }
}
\ar@{.} "A2";"A1"
\ar@{.} "A3";"A2"
\ar@{.} "A4";"A3"
\ar@{.} "A5";"A4"
\ar@{.}  "A6";"A5"
\ar@{-} "A6";"A1"
\ar@{.}@/_0.75pc/ "A1";"A2"
\ar@{.}@/_0.75pc/ "A2";"A3"
\ar@{.}@/_0.75pc/ "A3";"A4"
\ar@{.}@/_0.75pc/ "A4";"A5"
\ar@{.}@/_0.75pc/ "A5";"A6"
\ar@{-}@/_0.75pc/ "A6";"A1"
\ar@{.}"A3";"A5"
\ar@{.}@<-0.15ex> "A1";"A5"
\ar@{.}@<0.15ex> "A1";"A5"
\ar@{.}@<0.25ex>"A6";"A5"
\endxy&
\xy \xygraph{!{/r2pc/:} 
[] !P6"A"{~>{} }
}
\ar@{-} "A2";"A1"
\ar@{-} "A3";"A2"
\ar@{-} "A4";"A3"
\ar@{-} "A5";"A4"
\ar@{.}  "A6";"A5"
\ar@{.} "A6";"A1"
\ar@{.}@/_0.75pc/ "A1";"A2"
\ar@{.}@/_0.75pc/ "A2";"A3"
\ar@{.}@/_0.75pc/ "A3";"A4"
\ar@{.}@/_0.75pc/ "A4";"A5"
\ar@{.}@/_0.75pc/ "A5";"A6"
\ar@{.}@/_0.75pc/ "A6";"A1"
\ar@{.}"A3";"A5"
\ar@{.}@<-0.25ex> "A1";"A5"
\ar@{-} "A1";"A5"
\ar@{.}@<0.25ex>"A6";"A5"
\endxy
\end{array}
\end{array}
\]
\[
\begin{array}{ccc}
x^{11}y^5&&
\begin{array}{cc}
\xy \xygraph{!{/r2pc/:} 
[] !P6"A"{~>{} }
}
\ar@{.} "A2";"A1"
\ar@{.} "A3";"A2"
\ar@{.} "A4";"A3"
\ar@{.} "A5";"A4"
\ar@{.}  "A6";"A5"
\ar@{.} "A6";"A1"
\ar@{.}@/_0.75pc/ "A1";"A2"
\ar@{.}@/_0.75pc/ "A2";"A3"
\ar@{.}@/_0.75pc/ "A3";"A4"
\ar@{.}@/_0.75pc/ "A4";"A5"
\ar@{-}@/_0.75pc/ "A5";"A6"
\ar@{-}@/_0.75pc/ "A6";"A1"
\ar@{.}"A3";"A5"
\ar@{.}@<-0.25ex> "A1";"A5"
\ar@{-} "A1";"A5"
\ar@{.}@<0.25ex>"A6";"A5"
\endxy&
\xy \xygraph{!{/r2pc/:} 
[] !P6"A"{~>{} }
}
\ar@{-} "A2";"A1"
\ar@{-} "A3";"A2"
\ar@{-} "A4";"A3"
\ar@{-} "A5";"A4"
\ar@{.}  "A6";"A5"
\ar@{.} "A6";"A1"
\ar@{.}@/_0.75pc/ "A1";"A2"
\ar@{.}@/_0.75pc/ "A2";"A3"
\ar@{.}@/_0.75pc/ "A3";"A4"
\ar@{.}@/_0.75pc/ "A4";"A5"
\ar@{.}@/_0.75pc/ "A5";"A6"
\ar@{.}@/_0.75pc/ "A6";"A1"
\ar@{.}"A3";"A5"
\ar@{.}@<0.25ex> "A1";"A5"
\ar@{-} "A1";"A5"
\ar@{.}@<0.25ex>"A6";"A5"
\endxy
\end{array}\\
x^{3}y^8&&
\begin{array}{cccc}
\xy \xygraph{!{/r2pc/:} 
[] !P6"A"{~>{} }
}
\ar@{.} "A2";"A1"
\ar@{.} "A3";"A2"
\ar@{.} "A4";"A3"
\ar@{.} "A5";"A4"
\ar@{.}  "A6";"A5"
\ar@{.} "A6";"A1"
\ar@{.}@/_0.75pc/ "A1";"A2"
\ar@{.}@/_0.75pc/ "A2";"A3"
\ar@{.}@/_0.75pc/ "A3";"A4"
\ar@{.}@/_0.75pc/ "A4";"A5"
\ar@{-}@/_0.75pc/ "A5";"A6"
\ar@{-}@/_0.75pc/ "A6";"A1"
\ar@{.}"A3";"A5"
\ar@{.}@<0.25ex> "A1";"A5"
\ar@{-} "A1";"A5"
\ar@{.}@<0.25ex>"A6";"A5"
\endxy&
\xy \xygraph{!{/r2pc/:} 
[] !P6"A"{~>{} }
}
\ar@{-} "A2";"A1"
\ar@{.} "A3";"A2"
\ar@{.} "A4";"A3"
\ar@{.} "A5";"A4"
\ar@{.}  "A6";"A5"
\ar@{.} "A6";"A1"
\ar@{-}@/_0.75pc/ "A1";"A2"
\ar@{.}@/_0.75pc/ "A2";"A3"
\ar@{.}@/_0.75pc/ "A3";"A4"
\ar@{.}@/_0.75pc/ "A4";"A5"
\ar@{.}@/_0.75pc/ "A5";"A6"
\ar@{.}@/_0.75pc/ "A6";"A1"
\ar@{.}"A3";"A5"
\ar@{.}@<-0.15ex> "A1";"A5"
\ar@{.}@<0.15ex> "A1";"A5"
\ar@{.}@<0.25ex>"A6";"A5"
\endxy&
\xy \xygraph{!{/r2pc/:} 
[] !P6"A"{~>{} }
}
\ar@{.} "A2";"A1"
\ar@{-} "A3";"A2"
\ar@{.} "A4";"A3"
\ar@{.} "A5";"A4"
\ar@{.}  "A6";"A5"
\ar@{.} "A6";"A1"
\ar@{.}@/_0.75pc/ "A1";"A2"
\ar@{-}@/_0.75pc/ "A2";"A3"
\ar@{.}@/_0.75pc/ "A3";"A4"
\ar@{.}@/_0.75pc/ "A4";"A5"
\ar@{.}@/_0.75pc/ "A5";"A6"
\ar@{.}@/_0.75pc/ "A6";"A1"
\ar@{.}"A3";"A5"
\ar@{.}@<-0.15ex> "A1";"A5"
\ar@{.}@<0.15ex> "A1";"A5"
\ar@{.}@<0.25ex>"A6";"A5"
\endxy&\xy \xygraph{!{/r2pc/:} 
[] !P6"A"{~>{} }
}
\ar@{.} "A2";"A1"
\ar@{.} "A3";"A2"
\ar@{-} "A4";"A3"
\ar@{-} "A5";"A4"
\ar@{.}  "A6";"A5"
\ar@{.} "A6";"A1"
\ar@{.}@/_0.75pc/ "A1";"A2"
\ar@{.}@/_0.75pc/ "A2";"A3"
\ar@{.}@/_0.75pc/ "A3";"A4"
\ar@{.}@/_0.75pc/ "A4";"A5"
\ar@{.}@/_0.75pc/ "A5";"A6"
\ar@{.}@/_0.75pc/ "A6";"A1"
\ar@{-}"A3";"A5"
\ar@{.}@<-0.15ex> "A1";"A5"
\ar@{.}@<0.15ex> "A1";"A5"
\ar@{.}@<0.25ex>"A6";"A5"
\endxy
\end{array}\\
xy^{27}&&
\begin{array}{ccc}
\xy \xygraph{!{/r2pc/:} 
[] !P6"A"{~>{} }
}
\ar@{.} "A2";"A1"
\ar@{.} "A3";"A2"
\ar@{.} "A4";"A3"
\ar@{.} "A5";"A4"
\ar@{.}  "A6";"A5"
\ar@{.} "A6";"A1"
\ar@{-}@/_0.75pc/ "A1";"A2"
\ar@{-}@/_0.75pc/ "A2";"A3"
\ar@{.}@/_0.75pc/ "A3";"A4"
\ar@{.}@/_0.75pc/ "A4";"A5"
\ar@{-}@/_0.75pc/ "A5";"A6"
\ar@{-}@/_0.75pc/ "A6";"A1"
\ar@{-}"A3";"A5"
\ar@{.}@<-0.15ex> "A1";"A5"
\ar@{.}@<0.15ex> "A1";"A5"
\ar@{.}@<0.25ex>"A6";"A5"
\endxy& 
\xy \xygraph{!{/r2pc/:} 
[] !P6"A"{~>{} }
}
\ar@{.} "A2";"A1"
\ar@{.} "A3";"A2"
\ar@{-} "A4";"A3"
\ar@{.} "A5";"A4"
\ar@{.}  "A6";"A5"
\ar@{.} "A6";"A1"
\ar@{.}@/_0.75pc/ "A1";"A2"
\ar@{.}@/_0.75pc/ "A2";"A3"
\ar@{-}@/_0.75pc/ "A3";"A4"
\ar@{.}@/_0.75pc/ "A4";"A5"
\ar@{.}@/_0.75pc/ "A5";"A6"
\ar@{.}@/_0.75pc/ "A6";"A1"
\ar@{.}"A3";"A5"
\ar@{.}@<-0.15ex> "A1";"A5"
\ar@{.}@<0.15ex> "A1";"A5"
\ar@{.}@<0.25ex>"A6";"A5"
\endxy&
\xy \xygraph{!{/r2pc/:} 
[] !P6"A"{~>{} }
}
\ar@{.} "A2";"A1"
\ar@{.} "A3";"A2"
\ar@{.} "A4";"A3"
\ar@{-} "A5";"A4"
\ar@{.}  "A6";"A5"
\ar@{.} "A6";"A1"
\ar@{.}@/_0.75pc/ "A1";"A2"
\ar@{.}@/_0.75pc/ "A2";"A3"
\ar@{.}@/_0.75pc/ "A3";"A4"
\ar@{-}@/_0.75pc/ "A4";"A5"
\ar@{.}@/_0.75pc/ "A5";"A6"
\ar@{.}@/_0.75pc/ "A6";"A1"
\ar@{.}"A3";"A5"
\ar@{.}@<-0.15ex> "A1";"A5"
\ar@{.}@<0.15ex> "A1";"A5"
\ar@{.}@<0.25ex>"A6";"A5"
\endxy
\end{array}\\
y^{73}&&
\xy \xygraph{!{/r2pc/:} 
[] !P6"A"{~>{} }
}
\ar@{.} "A2";"A1"
\ar@{.} "A3";"A2"
\ar@{.} "A4";"A3"
\ar@{.} "A5";"A4"
\ar@{.} "A6";"A5"
\ar@{.} "A6";"A1"
\ar@{-}@/_0.75pc/ "A1";"A2"
\ar@{-}@/_0.75pc/ "A2";"A3"
\ar@{-}@/_0.75pc/ "A3";"A4"
\ar@{-}@/_0.75pc/ "A4";"A5"
\ar@{-}@/_0.75pc/ "A5";"A6"
\ar@{-}@/_0.75pc/ "A6";"A1"
\ar@{.}"A3";"A5"
\ar@{.}@<-0.15ex> "A1";"A5"
\ar@{.}@<0.15ex> "A1";"A5"
\ar@{.}@<0.25ex> "A6";"A5"
\endxy
\end{array}
\]
\end{example}
The proof of the following lemma follows the pattern in the above example: 
\begin{lemma}\label{loops}
For any $0\leq p\leq n$ we can view every $f\in
\t{Hom}(S_{i_p},S_{i_p})\cong \C{}[x,y]^G$ as a finite sum of
compositions of arrows in $D$ forming a cycle at
vertex $p$.
\end{lemma}
\begin{proof}
For every vertex we just need to justify that we can see the generators appearing in Lemma~\ref{generators} as a cycle.  To do this, we must introduce notation for the arrows in $D$, and in light of the map $\phi$ above we just label the arrows in $D$ by their notational counterpart in the reconstruction algebra.  To simplify the exposition we also consider cycles up to rotation.  For example by `$\cl{1}{0}\an{0}{1}$ at vertices 0 and 1' we actually mean `$\cl{1}{0}\an{0}{1}$ at vertex 1 and $\an{0}{1}\cl{1}{0}$ at vertex 0'; similarly by `$\CL{0}{0}$ at vertex $t$' we mean $\CL{t}{t}$ (i.e. we rotate suitably so that the cycle starts and finishes at the vertex we want).

Now we can always see $x^r$ at every vertex as $\CL{0}{0}$, and $y^r$ at every vertex as $\AN{0}{0}$.  We can always see $x^{r-a}y$ at vertices 0 and 1 as $\cl{1}{0}\an{0}{1}$; for the remaining vertices how to view $x^{r-a}y$ depends on the parameters.  If $\alpha_1>2$, then we can see $x^{r-a}y$ everywhere as $\CL{0}{1}k_1$.  If $\alpha_1=2$, then at all vertices $t$ with $1\leq t\leq \sigma_2$ we can see $x^{r-a}y$ as $\cl{t}{t-1}\an{t-1}{t}$, and for all $\sigma_2\leq t\leq n$ as $\CL{0}{\sigma_2}k_1$.  This takes care of $x^{r-a}y$, so we now move on to the remaining generators. 

\emph{Case $\alpha_1>2$.} If $\alpha_1=3$, then $x^{i_0-2i_1}y^{2j_1-j_0}$ can be seen at vertices 0 and 1 as $k_1\an{0}{1}$, at all vertices $t$ with $2\leq t\leq \sigma_2$ as $\cl{t}{t-1}\an{t-1}{t}$ and at all $\sigma_2\leq t\leq n$ as $\CL{0}{\sigma_2}k_{v_{\sigma_2}}$.  If $\alpha_1>3$, then for all $s$ with $2\leq s\leq \alpha_1-2$, $x^{i_0-si_1}y^{sj_1-j_0}$ can be seen at all vertices as $\CL{0}{1}k_s$.  Furthermore  $x^{i_0-(\alpha_1-1)i_1}y^{(\alpha_1-1)j_1-j_0}$ can be seen at vertices 0 and 1 as $k_{u_1}\an{0}{1}$, at all vertices $t$ with $2\leq t\leq \sigma_2$ as $\cl{t}{t-1}\an{t-1}{t}$, and at all $\sigma_2\leq t\leq n$ as $\CL{0}{\sigma_2}k_{v_{\sigma_2}}$.

\emph{Case $\alpha_q>2$ for $1<q<n$.}  Denote $\gamma_q=\t{min}\{ T:q<T\leq n\t{ with } \alpha_T>2 \}$, or take $\gamma_q=n$ if this set is empty.  Now if $\alpha_q=3$, then $x^{i_{q-1}-2i_q}y^{2j_q-j_{q-1}}$ can be seen at all vertices $t$ with $0\leq t\leq q$ as $\AN{0}{q}k_{u_q}$, at all vertices $t$ with $q+1\leq t\leq \gamma_2$ as $\cl{t}{t-1}\an{t-1}{t}$, and at all $\gamma_2\leq t\leq n$ as $\CL{0}{\gamma_q}k_{v_{\gamma_q}}$.  If $\alpha_q>3$, then for all $s$ with $2\leq s\leq \alpha_1-2$, $x^{i_{q-1}-si_q}y^{sj_{q}-j_{q-1}}$ can be seen at all vertices $t$ with $0\leq t\leq q$ as $\AN{0}{q}k_{v_q+s-2}$, and at all vertices $t$ with $q\leq t\leq n$ as $\CL{0}{q}k_{v_q+s-1}$.  Furthermore  $x^{i_{q-1}-(\alpha_1-1)i_q}y^{(\alpha_1-1)j_{q}-j_{q-1}}$ can be seen at all vertices $t$ with $0\leq t\leq q$ as $\AN{0}{q}k_{u_q}$, at all vertices $t$ with $q+1\leq t\leq \gamma_q$ as $\cl{t}{t-1}\an{t-1}{t}$, and at all $\gamma_2\leq t\leq n$ as $\CL{0}{\gamma_2}k_{v_{\gamma_2}}$.

\emph{Case $\alpha_n>2$.}  For all $s$ with $2\leq s\leq \alpha_n-1$, $x^{i_{n-1}-si_n}y^{sj_{n}-j_{n-1}}$ can be seen at all vertices $t$ as $\AN{0}{n}k_{v_n+(s-2)}$.
\end{proof}

\begin{lemma}\label{specials_can_be_seen}
For any $0\leq p\leq n$, every map $S_0\rightarrow S_{i_p}$ can be
seen in $D$.
\end{lemma}
\begin{proof}
The case $p=0$ is Lemma~\ref{loops}, so assume $p>0$. Then
$\t{Hom}(S_0,S_{i_p})\cong S_{i_p}$, which by
Theorem~\ref{Wunram_special_characteristation} is generated as a
module by $x^{i_p}$ and $y^{j_p}$.  Clearly both of these
generators are paths in $D$, and so since cycles at vertex
$p$ are all of $\C{}[x,y]^G$, we are done.
\end{proof}

Thus it remains to prove the following 2 statements:\\
\[
\begin{array}{cc}
\t{(i)} & \t{for any $0\leq q<p\leq n$, every map
$S_{i_p}\rightarrow S_{i_q}$ can be seen in $D$,}\\
\t{(ii)} & \t{for any $0< p<q\leq n$, every map
$S_{i_p}\rightarrow S_{i_q}$ can be seen in $D$.}
\end{array}
\]
We shall see that we need only prove (i), as appealing to
duality gives (ii) for free.  In what follows we refer to the
vertex $S_{{i_t}}$ as the $t^{th}$ vertex.
\begin{lemma}\label{factor_right_to_left}
For $0\leq q<p\leq n$, if $x^{z_1}y^{z_2}\in S_{i_q-i_p}$ with
$0\leq z_1,z_2\leq r-1$, then $x^{z_1}y^{z_2}$ factors as either\\
\t{(i)} $(x^{i_{p-1}-i_p})A$ for some $A\in S_{i_q-i_{p-1}}$  \\
\t{(ii)} $(x^{i_{p-1}-(t+1)i_p}y^{tj_p-j_{p-1}})B$ for some $B\in S_{i_q}$ and some $1\leq t\leq \alpha_p-2$ \\
\t{(iii)} $(y^{j_{p+1}-j_p})C$ for some $C\in S_{i_q-i_{p+1}}$.
\end{lemma}
\begin{proof}
The case $p=n$ is trivial, so assume $p<n$.  Clearly if $z_1\geq
i_{p-1}-i_p$, then we're in (i) so we can assume that $0\leq z_1<
i_{p-1}-i_p$. Consider the invariant $x^{z_1}y^{z_2+(j_p-j_q)}$.
Since we can see all invariants at every vertex, consider this as
a path in $D$ at the $p^{th}$ vertex.  It must leave the vertex,
and the hypothesis on $z_1$ means that it can't leave through the
$x^{i_{p-1}-i_p}$ map.

\emph{Case 1: $\alpha_p=2$} Then it must leave through the
$y^{j_{p+1}-j_p}$ map to vertex $p+1$, i.e.
\[
x^{z_1}y^{z_2+(j_p-j_q)}=y^{j_{p+1}-j_p}M
\]
for some path $M$ from vertex $p+1$ to $q$.  Now from vertex $p+1$
the path $M$ has to reach vertex $p$ again. But this can only be
reached in two ways, via the map $x^{i_{p-1}-i_p}=x^{i_p-i_{p+1}}$
from vertex $p+1$ to $p$, or via the map $y^{j_p-j_{p-1}}$ from
vertex $p-1$ to $p$. The hypothesis on the $x$ forces the latter,
so in particular the path factors through the 0 vertex.  It may be
true that there are cycles in the path that occur after the 0th
vertex, however, since we have all cycles at all vertices, we may move
these cycles to the 0th vertex and hence assume that the path $M$
finishes as the composition
\[
y^{j_1-j_0}y^{j_2-j_1}\hdots y^{j_p-j_{p-1}}=y^{j_{p}}.
\]
Hence we may write
\[
x^{z_1}y^{z_2+(j_p-j_q)}=y^{j_{p+1}-j_p}Ay^{j_p}
\]
for some path $A:S_{i_{p+1}}\rightarrow S_{0}$.  But since
$j_p\geq j_p-j_q$ we can cancel $y^{j_p-j_q}$ from both sides and
write
\[
x^{z_1}y^{z_2}=y^{j_{p+1}-j_p}A^\prime
\]
for some monomial $A^\prime$.  Necessarily $A^\prime\in
S_{i_q-i_{p+1}}$, and so we have the desired factorisation as in
(iii).

\emph{Case 2: $\alpha_p>2$.} For notational ease denote the extra
arrows leaving vertex $p$ by
$k_t=x^{i_{p-1}-(t+1)i_p}y^{tj_p-j_{p-1}}$. Now
$x^{z_1}y^{z_2+(j_p-j_q)}$ must leave vertex $p$ through the
$y^{j_{p+1}-j_p}$ map to vertex $p+1$, or through one of the extra
$k_t$.  We argue case by case:

Suppose first that $x^{z_1}y^{z_2+(j_p-j_q)}$ leaves through the
$y^{j_{p+1}-j_p}$ map to vertex $p+1$, i.e.
\[
x^{z_1}y^{z_2+(j_p-j_q)}=y^{j_{p+1}-j_p}M
\]
for some path $M$ from vertex $p+1$ to $p$.  If $M$ leaves vertex
$p+1$ through the $x^{i_{p}-i_{p+1}}$ map we are done, since then
\[
x^{z_1}y^{z_2+(j_p-j_q)}=x^{i_{p}-i_{p+1}}y^{j_{p+1}-j_p}M_1=k_{\alpha_p-2}y^{j_p}M_1
\]
for some monomial $M_1$, and so since $j_p\geq j_p-j_q$ we may
cancel and write $x^{z_1}y^{z_2}=k_{\alpha_p-2}M_1^\prime$ for
some monomial $M_1^\prime$ which necessarily belongs to $S_{i_q}$;
this is a factorisation as in (ii).  Hence we can assume that $M$
leaves vertex $p+1$ via another path. Since $p+1\leq n$ each of
these paths has $y$ component greater than or equal to
$y^{j_{p+1}-j_p}$, and so we may write
\[
x^{z_1}y^{z_2+(j_p-j_q)}=y^{2(j_{p+1}-j_p)}M_2
\]
for some monomial $M_2$.  But now $j_{p+1}-j_p> j_p-j_q$, so we may
cancel and write $x^{z_1}y^{z_2}=y^{j_{p+1}-j_p}M_2^\prime$ for
some monomial $M_2^\prime$ which necessarily belongs to
$S_{i_q-i_{p+1}}$; this is a factorisation as in (iii).

Now suppose that $x^{z_1}y^{z_2+(j_p-j_q)}$ factors through one of
the extra arrows $k_t$ out of vertex $p$. Thus
$x^{z_1}y^{z_2+(j_p-j_q)}=k_tB$ for some $1\leq t\leq \alpha_p-2$
and some path $B$ from $0$ to $p$. By
Lemma~\ref{specials_can_be_seen} there are 2 possibilities for
$B$: either $B=x^{i_p}B_1$ or $B=y^{j_p}B_2$ for some invariants
$B_1$ and $B_2$.  We split the remainder of the proof into cases
depending on the value of $t$:

If $t=1$, then $x^{z_1}y^{z_2+(j_p-j_q)}$ is either
$k_1x^{i_p}B_1=x^{i_{p-1}-i_p}y^{j_p-j_{p-1}}B_1$, which is
impossible by the assumption on $z_1$, or it's equal to
$k_1y^{j_p}B_2$.  But now $j_p-j_q\leq j_p$, and so after
cancelling we may write $x^{z_1}y^{z_2}=k_1B_2^\prime$ for some
monomial $B_2^\prime$ which necessarily belongs to $S_{i_q}$; this
gives a factorisation as in (ii).

The above argument takes care of $t=1$, and so we are done if
$\alpha_p=3$.  Hence the final case to consider is when
$\alpha_p>3$ and $t$ is such that $1<t\leq\alpha_p-2$.  Here
$x^{z_1}y^{z_2+(j_p-j_q)}$ is either
\[
k_tx^{i_p}B_1=k_{t-1}y^{j_p}B_1 \quad\t{or}\quad k_ty^{j_p}B_2.
\]
Again $j_p-j_q\leq j_p$, and so after cancelling we may write
$x^{z_1}y^{z_2}$ as either
\[
k_{t-1}B_1^\prime \quad\t{or}\quad k_tB_2^\prime
\]
for some monomials $B_2^\prime,B_2^\prime$ which necessarily
belong to $S_{i_q}$.  This gives the required factorisations as in
(ii), and completes the proof.
\end{proof}
The next two results are simple inductive arguments based on the
previous lemma.
\begin{cor}\label{case1_inductive_step}
For any $0\leq q <n$, every map $S_{i_n}\rightarrow S_{i_q}$ can
be seen in $D$.
\end{cor}
\begin{proof}
Let $x^{z_1}y^{z_2}\in S_{i_q-i_n}=S_{i_q-1}$, then by
Lemma~\ref{loops} we can remove cycles and so assume $0\leq
z_1,z_2\leq r-1$. By Lemma~\ref{factor_right_to_left} we know
$x^{z_1}y^{z_2}$ either factors \\
(i) through vertex ${n-1}$ as $(x^{i_{n-1}-i_n})A$ for some map
$A:S_{i_{n-1}}\rightarrow S_{i_q}$,\\
(ii) through vertex 0 as $(x^{i_{n-1}-(t+1)i_n}y^{tj_n-j_{n-1}})B$ for some $B: S_0\rightarrow S_{i_q}$ and some $1\leq t\leq \alpha_n-2$, \\
(iii) through  vertex 0 as $(y^{j_{n+1}-j_n})C$ for some $C\in
S_{i_q-i_{p+1}}=S_{i_q}=\t{Hom}(S_0,S_{i_q})$.\\
If we are in cases (ii) or (iii), then we're done by
Lemma~\ref{specials_can_be_seen} since we can see both $B$ and $C$
in the diagram $D$.  Hence assume case (i).  If $q=n-1$, then we are
done, since by Lemma~\ref{loops} we can see $A$ in the diagram $D$.
Hence we can assume that we are in case (i) with $q<n-1$. But now
by Lemma~\ref{factor_right_to_left} $A$ either factors\\
(i) through vertex ${n-2}$ as $(x^{i_{n-2}-i_{n-1}})A^\prime$ for
some map $A^\prime:S_{i_{n-2}}\rightarrow S_{i_q}$,\\
(ii) through vertex $0$ as $(x^{i_{n-2}-(t+1)i_{n-1}}y^{tj_{n-1}-j_{n-2}})B^\prime$ for some $B^\prime: S_0\rightarrow S_{i_q}$ and some $1\leq t\leq \alpha_{n-1}-2$, \\
(iii) through  vertex ${n}$ as $(y^{j_{n}-j_{n-1}})C^\prime$ for
some $C^\prime :S_n\rightarrow S_q$.\\
Since we removed cycles from $x^{z_1}y^{z_2}$ at the beginning, we
can't be in case (iii).  If we are in case (ii), then we are again done by
Lemma~\ref{specials_can_be_seen}, so we can again assume to be in case (i).
If $q=n-2$, then again we are done, so we can suppose we are in case
(i) with $q<n-2$.  Proceed inductively; since $0\leq q$, this
process must stop.
\end{proof}
\begin{cor}\label{case1}
For any $0\leq q<p\leq n$, every map $S_{i_p}\rightarrow S_{i_q}$
can be seen in $D$.
\end{cor}
\begin{proof}
Fix $q$.  This is now just a simple induction argument: if $p=n$,
then the result is true by Corollary~\ref{case1_inductive_step},
so let $p<n$ and assume the result is true for larger $p$.

Let $x^{z_1}y^{z_2}\in S_{i_q-i_p}$.  Then by Lemma~\ref{loops} we
can remove cycles and so assume $0\leq z_1,z_2\leq r-1$. By
Lemma~\ref{factor_right_to_left} we know
$x^{z_1}y^{z_2}$ either factors \\
(i) through vertex ${p-1}$ as $(x^{i_{p-1}-i_p})A$ for some map
$A:S_{i_{p-1}}\rightarrow S_{i_q}$,\\
(ii) through vertex 0 as $(x^{i_{p-1}-(t+1)i_p}y^{tj_p-j_{p-1}})B$ for some $B: S_0\rightarrow S_{i_q}$ and some $1\leq t\leq \alpha_p-2$, \\
(iii) through vertex $p+1$ as $(y^{j_{p+1}-j_p})C$ for some $C:
S_{i_{p+1}}\rightarrow S_{i_q}$.\\
If we are in case (iii), then by inductive hypothesis we can see
$C$ in the diagram $D$, and so we are done.  If we are in case (ii),  then by
Lemma~\ref{specials_can_be_seen} we are also finished.  Hence we
can assume we are in case (i).  If $q=p-1$, then we are done by
Lemma~\ref{loops} as we can see $A$ in the diagram $D$; hence we can
assume $q<p-1$.  Thus the result follows by an identical argument
as in Corollary~\ref{case1_inductive_step} above; we've removed
cycles so $A$ can't factor through $S_{i_p}$.
\end{proof}
To prove the corresponding statement in the opposite direction we
appeal to duality.  More precisely, the singularity defined by
$\frac{1}{r}(1,a)$ with $\frac{r}{a}=[\alpha_1,\hdots,\alpha_n]$
is isomorphic to the singularity $\frac{1}{r}(1,b)$ with
$\frac{r}{b}=[\alpha_n,\hdots,\alpha_1]$ (note $b=j_n$); the
isomorphism is given by swapping the $x$'s and $y$'s. To avoid
confusion we refer to everything for the singularity
$\frac{1}{r}(1,b)$ in typeface font; e.g.\ CM modules
$\tt{S}_t$, $i$-series by $\tt{i}$, diagram $\tt{D}$, etc.  The
explicit isomorphism is given by
\[
\begin{array}{rcl}
S_0 & \rightarrow & \tt{S}_0\\ x & \mapsto & \tt{y}\\
y& \mapsto& \tt{x}
\end{array}
\]
As in Lemma~\ref{read_backwards} flip the quiver vertex numbers by
the operation $^\prime$ which takes $0$ to itself (i.e.\ 
$0^\prime=0$), and reflects the other vertices in the natural line
of symmetry (i.e.\ $1^\prime =n$, $n^\prime=1$ etc).  Now for all
$1\leq p\leq n$ we have $i_p=\tt{j}_{p^\prime}$ and
$j_p=\tt{i}_{p^\prime}$, thus
\[
S_{i_p}=(x^{i_p},y^{j_p})S_0\cong
(\tt{y}^{\tt{j}_{p^\prime}},\tt{x}^{\tt{i}_{p^\prime}})\tt{S}_0=\tt{S}_{\tt{i}_{p^\prime}}.
\]
\begin{cor}\label{case2}
For the singularity $\frac{1}{r}(1,a)$, for any $0\leq p<q\leq n$,
every map $S_{i_p}\rightarrow S_{i_q}$ can be seen in 
$D$.
\end{cor}
\begin{proof}
Under the duality above, $x^{z_1}y^{z_2}:S_{i_p}\rightarrow
S_{i_q}$ corresponds to
$\tt{y}^{z_1}\tt{x}^{z_2}:\tt{S}_{\tt{i}_{p^\prime}}\rightarrow
\tt{S}_{\tt{i}_{q^\prime}}$. But since $q^\prime<p^\prime$ we can
by Corollary~\ref{case1} view this in the diagram $\tt{D}$ as the
composition of monomials whose powers are in terms of $\tt{i}$'s,
$\tt{j}$'s and $\alpha_t$'s.  Under the duality isomorphisms we
can view this as a path in the diagram $D$.
\end{proof}

Summarizing Lemma~\ref{loops}, Corollary~\ref{case1} and
Corollary~\ref{case2} we have:
\begin{prop}\label{map_onto} For any $\frac{1}{r}(1,a)$, for any
$0\leq p,q\leq n$, we can see every map $S_{i_p}\rightarrow S_{i_q}$
in the diagram $D$. Consequently the natural map $\C{}Q\rightarrow \t{End}_R(\oplus_{p=1}^{n+1}S_{i_p})$ is surjective.
\end{prop}

The above is extremely useful if we want to compute some examples.  One way to compute the
endomorphism ring of the specials is to take the McKay quiver and
corner (i.e.\ ignore a vertex and compose maps that pass through that
vertex) the non-special vertices.  Of course, the larger the group
the longer this computation; for the example $\frac{1}{40}(1,11)$
there are forty vertices in the McKay quiver, and we need to corner thirty-six of them. Given any example $\frac{1}{r}(1,a)$,
Proposition~\ref{map_onto} saves us this long computation since
the algorithm to produce the necessary diagram involves only the
continued fraction expansion of $\frac{r}{a}$ and the associated
$i$ and $j$ series, all of which are extremely quick to calculate.
\begin{example}\label{1/40(1,11)}
\t{For $\frac{1}{40}(1,11)$, $\frac{40}{11}=[4,3,4]$, so the $i$-
and $j$-series are
\[
\begin{array}{ccccccccc}
 i_0=40 &>& i_1=11 &>& i_2=4 &>& i_3=1 &>& i_4=0, \\
 j_0=0 &<& j_1=1 &<& j_2=4 &<& j_3=11 &<& j_4=40. \end{array}\]
By Proposition~\ref{map_onto} the endomorphism ring of the
specials is
\[
\xymatrix@R=40pt@C=40pt{S_1\ar|{x^3}[r]\ar@/_1.5pc/|{x^2y^7}[d]\ar@/_3pc/|{xy^{18}}[d]\ar@/_4.5pc/|{y^{29}}[d]
&
S_4\ar|{x^7}[d]\ar@/_1.5pc/_{y^7}[l]\ar|{x^3y^3}[dl] \\
S_0\ar[u]|{x}\ar@<-1ex>@/_1.5pc/_{y}[r]  &
S_{11}\ar@<1.5ex>|{x^{29}}[l]\ar|{x^{18}y}[l]\ar|{x^7y^2}@<-1.5ex>[l]\ar@/_1.5pc/_{y^3}[u]}
\]
Notice the correspondence with Example~\ref{4,3,4}.}
\end{example}
\begin{example}
\t{For the group $\frac{1}{693}(1,256)$,
$\frac{693}{256}=[3,4,2,4,2,3,3]$, so the $i$- and $j$- series are}
\[
\begin{array}{c|ccccccccc}
&0 & 1&2&3&4&5&6&7&8\\ \hline
i &693& 256 & 75 & 44 & 13 & 8 & 3 & 1 & 0\\
j& 0& 1& 3&11&19&65&111&268&693.
\end{array}
\]
\t{and, further, the endomorphism ring of the specials is}
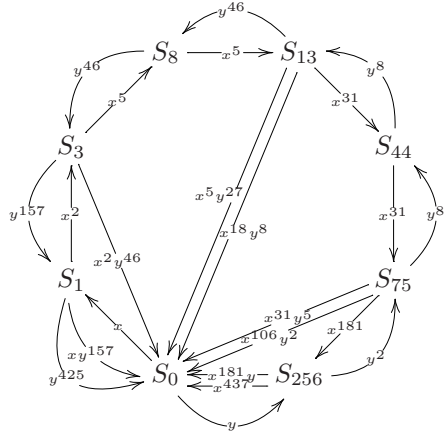
\[
\def\alphanum{\ifcase\xypolynode\or S_{44}\or S_{13}\or S_{8}\or S_{3}\or S_1
\or S_{0}\or S_{256}\or S_{75}  \fi}
\xy \xygraph{!{/r5.5pc/:} 
[] !P8"A"{~>{} ~*{\alphanum}}
}
\ar|{{}_{x^{31}}} "A2";"A1"
\ar|{{}_{x^5}} "A3";"A2"
\ar|{{}_{x^5}} "A4";"A3"
\ar|{{}_{x^2}} "A5";"A4"
\ar|{{}_{x}}  "A6";"A5"
\ar|{{}_{x^{31}}} "A1";"A8"
\ar|{{}_{x^{181}}} "A8";"A7"
\ar@<1ex>|{{}_{x^{437}}} "A7";"A6"
\ar@/_1.35pc/|{{}_{y^{8}}} "A1";"A2"
\ar@/_1.35pc/|{{}_{y^{46}}} "A2";"A3"
\ar@/_1.35pc/|{{}_{y^{46}}} "A3";"A4"
\ar@/_1.35pc/|{{}_{y^{157}}} "A4";"A5"
\ar@/_2.25pc/|{{}_{y^{425}}} "A5";"A6"
\ar@<-0.5ex>@/_1.45pc/|{{}_{y}} "A6";"A7"
\ar@/_1.35pc/|{{}_{y^{2}}} "A7";"A8"
\ar@/_1.35pc/|{{}_{y^{8}}} "A8";"A1"

\ar|{{}_{x^{181}y}} "A7";"A6"
\ar@<-0.5ex>|(0.45){{}_{x^{31}y^5}} "A8";"A6"
\ar@<0.5ex>|(0.55){{}_{x^{106}y^2}} "A8";"A6"
\ar@<-0.5ex>|(0.45){{}_{x^{5}y^{27}}\quad} "A2";"A6"
\ar@<0.5ex>|(0.55){\quad{}_{x^{18}y^8}} "A2";"A6"
\ar|{{}_{x^{2}y^{46}}} "A4";"A6"
\ar@<-0.5ex>@/_1pc/|{{}_{xy^{157}}} "A5";"A6"
\endxy
\]
\end{example}

In Proposition~\ref{map_onto} above we have shown that the natural map $\phi:\C{}Q\rightarrow \t{End}_R(\oplus_{p=1}^{n+1}S_{i_p})$ is surjective, and so we now show that the kernel is generated by the reconstruction algebra relations.  To achieve this, double index the arrows in $A_{r,a}$ as follows:
\[
\begin{array}{c|c}
\t{arrow} & \t{double index} \\ \hline
\cl{0}{n} &  (1,0) \\
\cl{t}{t-1} & (i_{t-1}-i_{t},0)\\
\an{n}{0} & (0,r-j_n) \\
\an{t}{t+1} & (0,j_{t+1}-j_t)\\
k_s &
(i_{l_s-1}-((s-V_{l_s})+1)i_{l_s},(s-V_{l_s})j_{l_s}-j_{l_s-1})
\end{array}
\]
It is easy to see that the two terms in any relation for $A_{r,a}$
have the same double index, and so the double index can be extended
to all paths in $A_{r,a}$.  We shall now show that if there exists
a path of double index $(z_1,z_2)$ leaving a vertex $t$ in
$A_{r,a}$, then the path is necessarily unique.
\begin{defin}
For a given vertex $t$ in $A_{r,a}$ define the web of paths
leaving $t$ as follows: place $t$ in the $(0,0)$ position of a
2-dimensional grid, and for each arrow leaving $t$ draw a line
from $(0,0)$ to the double index of that arrow.  Mark the end of
this line by the head of the arrow.  Continue in this way for all
the heads of the arrows.
\end{defin}
This is best understood after consulting
some examples.  In the following two examples the web should be
extended forever in the obvious direction; for practical purposes
we draw only the start of the picture.
\begin{example} \t{The web of paths from vertex $0$ in $A_{4,1}$ and $A_{11,3}$ begins respectively:}
{\tiny{
\[
\begin{array}{cc}
\begin{array}{c}
\xymatrix@R=15pt@C=15pt{0\ar[0,1]|(0.35){a_2}\ar[d]|(0.35){a_1} &
1\ar[3,0]|(0.45){c_1}
\ar[0,3]|{c_2}\ar[1,2]|{k_1}\ar[2,1]|(0.4){k_2} & & & 0\ar[d]|(0.35){a_1}\\
1\ar[3,0]|(0.45){c_1}
\ar[0,3]|(0.2){c_2}\ar[1,2]|(0.3){k_1}\ar[2,1]|(0.4){k_2} & & & 0\ar[0,1]|(0.35){a_2}\ar[d]|(0.35){a_1} & 1\\ &  &0\ar[0,1]|(0.35){a_2}\ar[d]|(0.35){a_1}&1&\\
&0\ar[d]|(0.35){a_1}\ar[0,1]|(0.35){a_2}&1&&\\0\ar[0,1]|(0.35){a_2}&1
&&&}
\end{array}
&\begin{array}{c}
\xymatrix@R=10pt@C=10pt{
0\ar[0,1]|(0.35){\an{0}{1}}\ar[d]|(0.3){\cl{0}{2}}
& 1\ar[8,0]|(0.7){\cl{1}{0}}\ar[5,1]|(0.5){k_1}\ar[2,2]|(0.3){k_2}\ar[0,3]|{\an{1}{2}} & & & 2\ar[1,3]|{k_3}\ar[2,0]|(0.3){\cl{2}{1}}\ar[0,7]|{\an{2}{0}} &&&&&&&0\ar[d]|(0.3){\cl{0}{2}}\\
2\ar[1,3]|(0.2){k_3}\ar[2,0]|{\cl{2}{1}}\ar[0,7]|{\an{2}{0}} & & &
&&&&0\ar[0,1]|(0.35){\an{0}{1}}\ar[d]|(0.3){\cl{0}{2}} &
1\ar[0,3]|{\an{1}{2}}&&&2\\
&&&0\ar[d]|(0.3){\cl{0}{2}}\ar[0,1]|(0.35){\an{0}{1}}&1\ar[0,3]|{\an{1}{2}} &&&2&&&&\\
1\ar[8,0]|{\cl{1}{0}}\ar[5,1]|(0.6){k_1}\ar[2,2]|(0.3){k_2}\ar[0,3]|(0.2){\an{1}{2}}&&&2\ar[2,0]|{\cl{2}{1}}&&&&&&&&\\
&&&&&&&&&&&\\
&&0\ar[d]|(0.3){\cl{0}{2}}\ar[0,1]|(0.35){\an{0}{1}}&1&&&&&&&&\\
&&2\ar[2,0]|{\cl{2}{1}}&&&&&&&&&\\
&&&&&&&&&&&\\
&0\ar[d]|(0.3){\cl{0}{2}}\ar[0,1]|(0.35){\an{0}{1}}&1&&&&&&&&&\\
&2\ar[2,0]|{\cl{2}{1}}&&&&&&&&&&\\
&&&&&&&&&&&\\
0\ar[0,1]|(0.35){\an{0}{1}}&1&&&&&&&&&&\\ }
\end{array}
\end{array}
\]}}
\end{example}
We call the points in the web of paths that lie in the set
$\{(w,0):0\leq w<n \}$ the left rail, and similarly those that lie
in the set $\{(0,w):0\leq w<n \}$ are called the top rail.
\begin{defin}
Draw the left rail and the top rail in the web of paths leaving
vertex $0$, and draw in every arrow leaving these vertices.  Join
the ends of these by using only vertical and horizontal paths, and
call the resulting diagram $F$.
\end{defin}
The fact that this can always be done is due to the grading we put
on $A_{r,a}$, together with simple combinatorics with continued
fractions.  The examples above show $F$ for $A_{4,1}$ and
$A_{11,3}$.

Now $F$ generates the web of paths leaving $0$
in the sense that all paths can be obtained by gluing on extra
copies of $F$ to the existing copy, as in the following picture:
\[
\xy
\ar@{-}(0,0);(22,0)
\ar@{-}(14,-2);(22,-2)
\ar@{-}(6,-4);(14,-4)
\ar@{-}(4,-10);(6,-10)
\ar@{-}(2,-16);(4,-16)
\ar@{-}(0,-22);(2,-22)

\ar@{-}(22,0);(22,-2)
\ar@{-}(14,-2);(14,-4)
\ar@{-}(6,-4);(6,-10)
\ar@{-}(4,-10);(4,-16)
\ar@{-}(2,-16);(2,-22)
\ar@{-}(0,0);(0,-22)

\ar@{.}(0,-2);(14,-2)
\ar@{.}(0,-6);(6,-6)

\ar@{.}(2,0);(2,-16)
\ar@{.}(8,0);(8,-4)

\ar@{.}(8,0);(14,-2)
\ar@{.}(2,0);(4,-10)
\ar@{.}(2,0);(6,-4)
\ar@{.}(0,-2);(6,-4)
\ar@{.}(0,-6);(4,-10)
\ar@{.}(0,-6);(2,-16)

\ar(24,-23);(24,-1)
\ar(23.8,-23.2);(16,-4)
\ar(23.6,-23.4);(8,-6)
\ar(23.4,-23.6);(6,-12)
\ar(23.2,-23.8);(4,-18)
\ar(23,-24);(2,-23)

\ar@{-}(24,-24);(46,-24)
\ar@{-}(38,-26);(46,-26)
\ar@{-}(30,-28);(38,-28)
\ar@{-}(28,-34);(30,-34)
\ar@{-}(26,-40);(28,-40)
\ar@{-}(24,-46);(26,-46)

\ar@{-}(46,-24);(46,-26)
\ar@{-}(38,-26);(38,-28)
\ar@{-}(30,-28);(30,-34)
\ar@{-}(28,-34);(28,-40)
\ar@{-}(26,-40);(26,-46)
\ar@{-}(24,-24);(24,-46)

\ar@{.}(24,-26);(38,-26)
\ar@{.}(24,-30);(30,-30)

\ar@{.}(26,-24);(26,-40)
\ar@{.}(32,-24);(32,-28)

\ar@{.}(32,-24);(38,-26)
\ar@{.}(26,-24);(28,-34)
\ar@{.}(26,-24);(30,-28)
\ar@{.}(24,-26);(30,-28)
\ar@{.}(24,-30);(28,-34)
\ar@{.}(24,-30);(26,-40)

\endxy
\]
The copies of $F$ glue together seamlessly due to the symmetry and repetition inside $F$.  The boundary of $F$ consists entirely of straight lines, and crucially $F$
contains (by definition) all paths from the rails so there can be
no paths that leap over the boundary to create new paths.

Since $F$ generates the web of paths leaving $0$, it is clear
(since $F$ can) that the web of paths can be subdivided into small
`squares'; we call these \emph{basic squares}.
\begin{defin}
A square is a pair of paths $(p_1\hdots p_s,q_1\hdots q_t)$ with
$tail(p_1)=tail(q_1)$ and $head(p_s)=head(q_t)$.  A square in $F$
is called basic if $p_i \notin \{ q_1,\hdots,q_t\}$ for all $1\leq
i\leq s$ and $q_j \notin \{ p_1,\hdots,p_s\}$ for all $1\leq j\leq
t$.
\end{defin}

By the definition and structure of $F$ it is clear that if all
basic squares in $F$ commute, then all squares in $F$ commute.
Since $F$ generates the web of paths, this means all squares
commute (since they are made from squares in $F$), giving us the
required uniqueness of path.

Because of the symmetry in $F$ there are in fact repetitions of
the basic squares inside $F$.  More precisely, the basic squares
starting at the $1$ on the top rail are the same as those starting
at $1$ on the left rail, etc.  Thus by the symmetry in $F$ it is
clear that all the basic squares leaving the left rail are
all the basic squares in $F$.

\begin{example}
 \t{The 6 basic squares in the example $A_{4,1}$ above are indicated by the following solid lines:
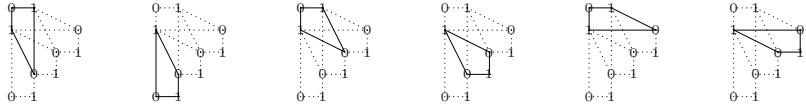
\[
\begin{array}{cccccc}
\begin{array}{c}
\def\objectstyle{\scriptscriptstyle}
\xy0;/r.35pc/:
\POS(0,0)*+{0},(2,0)*+{1},(0,-2)*+{1},(6,-2)*+{0},(4,-4)*+{0},(6,-4)*+{1},(2,-6)*+{0},(4,-6)*+{1},(0,-8)*+{0},(2,-8)*+{1},
\ar@{-}(0,0);(0,-2)
\ar@{-}(0,0);(2,0)
\ar@{.}(2,0);(6,-2)
\ar@{.}(2,0);(4,-4)
\ar@{-}(2,0);(2,-6)
\ar@{.}(0,-2);(6,-2)
\ar@{.}(0,-2);(4,-4)
\ar@{-}(0,-2);(2,-6)
\ar@{.}(0,-2);(0,-8)
\ar@{.}(6,-2);(6,-4)
\ar@{.}(4,-4);(4,-6)
\ar@{.}(2,-6);(2,-8)
\ar@{.}(0,-8);(2,-8)
\ar@{.}(2,-6);(4,-6)
\ar@{.}(4,-4);(6,-4)
\endxy
\end{array}&
\begin{array}{c}
\def\objectstyle{\scriptscriptstyle}
\xy0;/r.35pc/:
\POS(0,0)*+{0},(2,0)*+{1},(0,-2)*+{1},(6,-2)*+{0},(4,-4)*+{0},(6,-4)*+{1},(2,-6)*+{0},(4,-6)*+{1},(0,-8)*+{0},(2,-8)*+{1},
\ar@{.}(0,0);(0,-2)
\ar@{.}(0,0);(2,0)
\ar@{.}(2,0);(6,-2)
\ar@{.}(2,0);(4,-4)
\ar@{.}(2,0);(2,-6)
\ar@{.}(0,-2);(6,-2)
\ar@{.}(0,-2);(4,-4)
\ar@{-}(0,-2);(2,-6)
\ar@{-}(0,-2);(0,-8)
\ar@{.}(6,-2);(6,-4)
\ar@{.}(4,-4);(4,-6)
\ar@{-}(2,-6);(2,-8)
\ar@{-}(0,-8);(2,-8)
\ar@{.}(2,-6);(4,-6)
\ar@{.}(4,-4);(6,-4)
\endxy
\end{array}&
\begin{array}{c}
\def\objectstyle{\scriptscriptstyle}
\xy0;/r.35pc/:
\POS(0,0)*+{0},(2,0)*+{1},(0,-2)*+{1},(6,-2)*+{0},(4,-4)*+{0},(6,-4)*+{1},(2,-6)*+{0},(4,-6)*+{1},(0,-8)*+{0},(2,-8)*+{1},
\ar@{-}(0,0);(0,-2)
\ar@{-}(0,0);(2,0)
\ar@{.}(2,0);(6,-2)
\ar@{-}(2,0);(4,-4)
\ar@{.}(2,0);(2,-6)
\ar@{.}(0,-2);(6,-2)
\ar@{-}(0,-2);(4,-4)
\ar@{.}(0,-2);(2,-6)
\ar@{.}(0,-2);(0,-8)
\ar@{.}(6,-2);(6,-4)
\ar@{.}(4,-4);(4,-6)
\ar@{.}(2,-6);(2,-8)
\ar@{.}(0,-8);(2,-8)
\ar@{.}(2,-6);(4,-6)
\ar@{.}(4,-4);(6,-4)
\endxy
\end{array}&
\begin{array}{c}
\def\objectstyle{\scriptscriptstyle}
\xy0;/r.35pc/:
\POS(0,0)*+{0},(2,0)*+{1},(0,-2)*+{1},(6,-2)*+{0},(4,-4)*+{0},(6,-4)*+{1},(2,-6)*+{0},(4,-6)*+{1},(0,-8)*+{0},(2,-8)*+{1},
\ar@{.}(0,0);(0,-2)
\ar@{.}(0,0);(2,0)
\ar@{.}(2,0);(6,-2)
\ar@{.}(2,0);(4,-4)
\ar@{.}(2,0);(2,-6)
\ar@{.}(0,-2);(6,-2)
\ar@{-}(0,-2);(4,-4)
\ar@{-}(0,-2);(2,-6)
\ar@{.}(0,-2);(0,-8)
\ar@{.}(6,-2);(6,-4)
\ar@{-}(4,-4);(4,-6)
\ar@{.}(2,-6);(2,-8)
\ar@{.}(0,-8);(2,-8)
\ar@{-}(2,-6);(4,-6)
\ar@{.}(4,-4);(6,-4)
\endxy
\end{array}&
\begin{array}{c}
\def\objectstyle{\scriptscriptstyle}
\xy0;/r.35pc/:
\POS(0,0)*+{0},(2,0)*+{1},(0,-2)*+{1},(6,-2)*+{0},(4,-4)*+{0},(6,-4)*+{1},(2,-6)*+{0},(4,-6)*+{1},(0,-8)*+{0},(2,-8)*+{1},
\ar@{-}(0,0);(0,-2)
\ar@{-}(0,0);(2,0)
\ar@{-}(2,0);(6,-2)
\ar@{.}(2,0);(4,-4)
\ar@{.}(2,0);(2,-6)
\ar@{-}(0,-2);(6,-2)
\ar@{.}(0,-2);(4,-4)
\ar@{.}(0,-2);(2,-6)
\ar@{.}(0,-2);(0,-8)
\ar@{.}(6,-2);(6,-4)
\ar@{.}(4,-4);(4,-6)
\ar@{.}(2,-6);(2,-8)
\ar@{.}(0,-8);(2,-8)
\ar@{.}(2,-6);(4,-6)
\ar@{.}(4,-4);(6,-4)
\endxy
\end{array}&
\begin{array}{c}
\def\objectstyle{\scriptscriptstyle}
\xy0;/r.35pc/:
\POS(0,0)*+{0},(2,0)*+{1},(0,-2)*+{1},(6,-2)*+{0},(4,-4)*+{0},(6,-4)*+{1},(2,-6)*+{0},(4,-6)*+{1},(0,-8)*+{0},(2,-8)*+{1},
\ar@{.}(0,0);(0,-2)
\ar@{.}(0,0);(2,0)
\ar@{.}(2,0);(6,-2)
\ar@{.}(2,0);(4,-4)
\ar@{.}(2,0);(2,-6)
\ar@{-}(0,-2);(6,-2)
\ar@{-}(0,-2);(4,-4)
\ar@{.}(0,-2);(2,-6)
\ar@{.}(0,-2);(0,-8)
\ar@{-}(6,-2);(6,-4)
\ar@{.}(4,-4);(4,-6)
\ar@{.}(2,-6);(2,-8)
\ar@{.}(0,-8);(2,-8)
\ar@{.}(2,-6);(4,-6)
\ar@{-}(4,-4);(6,-4)
\endxy
\end{array}
\end{array}
\]
These are precisely the relations.  Note that these prove that the
paths $a_2k_1a_1$ and $a_1c_1k_2$ coincide, since that square can be
subdivided into 2 basic squares, both of which commute.}
\end{example}
\begin{example}\label{basic_squares_A11,3}
 \t{The 9 basic squares in the example $A_{11,3}$ above are
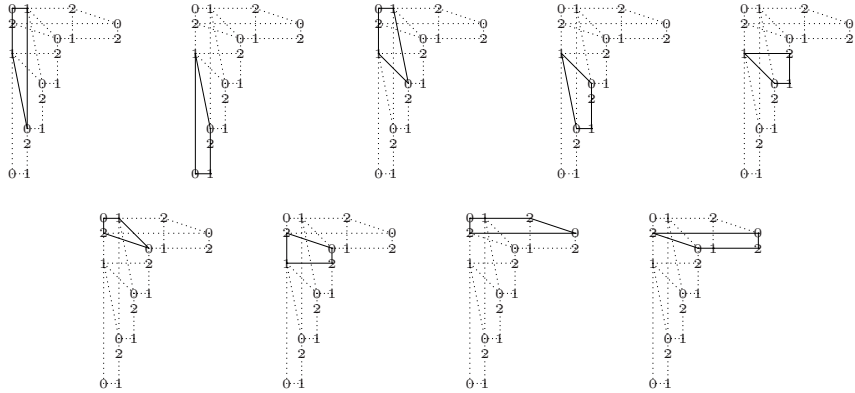
\[
\begin{array}{ccccc}
\begin{array}{c}
\def\objectstyle{\scriptscriptstyle}
\xy
\POS(0,0)*+{0},(2,0)*+{1},(8,0)*+{2},(0,-2)*+{2},(0,-6)*+{1},(0,-22)*+{0},(2,-22)*+{1},(2,-16)*+{0},(2,-18)*+{2},(4,-16)*+{1},(4,-10)*+{0},(4,-12)*+{2},(6,-10)*+{1},(6,-4)*+{0},(6,-6)*+{2},(8,-4)*+{1},(14,-2)*+{0},(14,-4)*+{2},
\ar@{-}(0,0);(2,0)
\ar@{-}(0,0);(0,-2)
\ar@{-}(2,0);(2,-16)
\ar@{.}(2,0);(4,-10)
\ar@{.}(2,0);(6,-4)
\ar@{.}(2,0);(8,0)
\ar@{.}(8,0);(8,-4)
\ar@{.}(8,0);(14,-2)
\ar@{.}(0,-2);(14,-2)
\ar@{.}(0,-2);(6,-4)
\ar@{-}(0,-2);(0,-6)
\ar@{.}(0,-6);(0,-22)
\ar@{-}(0,-6);(2,-16)
\ar@{.}(0,-6);(4,-10)
\ar@{.}(0,-6);(6,-6)
\ar@{.}(6,-4);(8,-4)
\ar@{.}(8,-4);(14,-4)
\ar@{.}(4,-10);(6,-10)
\ar@{.}(2,-16);(4,-16)
\ar@{.}(0,-22);(2,-22)
\ar@{.}(14,-2);(14,-4)
\ar@{.}(6,-4);(6,-6)
\ar@{.}(6,-6);(6,-10)
\ar@{.}(4,-10);(4,-12)
\ar@{.}(4,-12);(4,-16)
\ar@{.}(2,-16);(2,-18)
\ar@{.}(2,-18);(2,-22)
\endxy
\end{array}
&
\begin{array}{c}
\def\objectstyle{\scriptscriptstyle}
\xy
\POS(0,0)*+{0},(2,0)*+{1},(8,0)*+{2},(0,-2)*+{2},(0,-6)*+{1},(0,-22)*+{0},(2,-22)*+{1},(2,-16)*+{0},(2,-18)*+{2},(4,-16)*+{1},(4,-10)*+{0},(4,-12)*+{2},(6,-10)*+{1},(6,-4)*+{0},(6,-6)*+{2},(8,-4)*+{1},(14,-2)*+{0},(14,-4)*+{2},
\ar@{.}(0,0);(2,0)
\ar@{.}(0,0);(0,-2)
\ar@{.}(2,0);(2,-16)
\ar@{.}(2,0);(4,-10)
\ar@{.}(2,0);(6,-4)
\ar@{.}(2,0);(8,0)
\ar@{.}(8,0);(8,-4)
\ar@{.}(8,0);(14,-2)
\ar@{.}(0,-2);(14,-2)
\ar@{.}(0,-2);(6,-4)
\ar@{.}(0,-2);(0,-6)
\ar@{-}(0,-6);(0,-22)
\ar@{-}(0,-6);(2,-16)
\ar@{.}(0,-6);(4,-10)
\ar@{.}(0,-6);(6,-6)
\ar@{.}(6,-4);(8,-4)
\ar@{.}(8,-4);(14,-4)
\ar@{.}(4,-10);(6,-10)
\ar@{.}(2,-16);(4,-16)
\ar@{-}(0,-22);(2,-22)
\ar@{.}(14,-2);(14,-4)
\ar@{.}(6,-4);(6,-6)
\ar@{.}(6,-6);(6,-10)
\ar@{.}(4,-10);(4,-12)
\ar@{.}(4,-12);(4,-16)
\ar@{-}(2,-16);(2,-18)
\ar@{-}(2,-18);(2,-22)
\endxy
\end{array}
&
\begin{array}{c}
\def\objectstyle{\scriptscriptstyle}
\xy
\POS(0,0)*+{0},(2,0)*+{1},(8,0)*+{2},(0,-2)*+{2},(0,-6)*+{1},(0,-22)*+{0},(2,-22)*+{1},(2,-16)*+{0},(2,-18)*+{2},(4,-16)*+{1},(4,-10)*+{0},(4,-12)*+{2},(6,-10)*+{1},(6,-4)*+{0},(6,-6)*+{2},(8,-4)*+{1},(14,-2)*+{0},(14,-4)*+{2},
\ar@{-}(0,0);(2,0)
\ar@{-}(0,0);(0,-2)
\ar@{.}(2,0);(2,-16)
\ar@{-}(2,0);(4,-10)
\ar@{.}(2,0);(6,-4)
\ar@{.}(2,0);(8,0)
\ar@{.}(8,0);(8,-4)
\ar@{.}(8,0);(14,-2)
\ar@{.}(0,-2);(14,-2)
\ar@{.}(0,-2);(6,-4)
\ar@{-}(0,-2);(0,-6)
\ar@{.}(0,-6);(0,-22)
\ar@{.}(0,-6);(2,-16)
\ar@{-}(0,-6);(4,-10)
\ar@{.}(0,-6);(6,-6)
\ar@{.}(6,-4);(8,-4)
\ar@{.}(8,-4);(14,-4)
\ar@{.}(4,-10);(6,-10)
\ar@{.}(2,-16);(4,-16)
\ar@{.}(0,-22);(2,-22)
\ar@{.}(14,-2);(14,-4)
\ar@{.}(6,-4);(6,-6)
\ar@{.}(6,-6);(6,-10)
\ar@{.}(4,-10);(4,-12)
\ar@{.}(4,-12);(4,-16)
\ar@{.}(2,-16);(2,-18)
\ar@{.}(2,-18);(2,-22)
\endxy
\end{array}
&\begin{array}{c}
\def\objectstyle{\scriptscriptstyle}
\xy
\POS(0,0)*+{0},(2,0)*+{1},(8,0)*+{2},(0,-2)*+{2},(0,-6)*+{1},(0,-22)*+{0},(2,-22)*+{1},(2,-16)*+{0},(2,-18)*+{2},(4,-16)*+{1},(4,-10)*+{0},(4,-12)*+{2},(6,-10)*+{1},(6,-4)*+{0},(6,-6)*+{2},(8,-4)*+{1},(14,-2)*+{0},(14,-4)*+{2},
\ar@{.}(0,0);(2,0)
\ar@{.}(0,0);(0,-2)
\ar@{.}(2,0);(2,-16)
\ar@{.}(2,0);(4,-10)
\ar@{.}(2,0);(6,-4)
\ar@{.}(2,0);(8,0)
\ar@{.}(8,0);(8,-4)
\ar@{.}(8,0);(14,-2)
\ar@{.}(0,-2);(14,-2)
\ar@{.}(0,-2);(6,-4)
\ar@{.}(0,-2);(0,-6)
\ar@{.}(0,-6);(0,-22)
\ar@{-}(0,-6);(2,-16)
\ar@{-}(0,-6);(4,-10)
\ar@{.}(0,-6);(6,-6)
\ar@{.}(6,-4);(8,-4)
\ar@{.}(8,-4);(14,-4)
\ar@{.}(4,-10);(6,-10)
\ar@{-}(2,-16);(4,-16)
\ar@{.}(0,-22);(2,-22)
\ar@{.}(14,-2);(14,-4)
\ar@{.}(6,-4);(6,-6)
\ar@{.}(6,-6);(6,-10)
\ar@{-}(4,-10);(4,-12)
\ar@{-}(4,-12);(4,-16)
\ar@{.}(2,-16);(2,-18)
\ar@{.}(2,-18);(2,-22)
\endxy
\end{array}
&\begin{array}{c}
\def\objectstyle{\scriptscriptstyle}
\xy
\POS(0,0)*+{0},(2,0)*+{1},(8,0)*+{2},(0,-2)*+{2},(0,-6)*+{1},(0,-22)*+{0},(2,-22)*+{1},(2,-16)*+{0},(2,-18)*+{2},(4,-16)*+{1},(4,-10)*+{0},(4,-12)*+{2},(6,-10)*+{1},(6,-4)*+{0},(6,-6)*+{2},(8,-4)*+{1},(14,-2)*+{0},(14,-4)*+{2},
\ar@{.}(0,0);(2,0)
\ar@{.}(0,0);(0,-2)
\ar@{.}(2,0);(2,-16)
\ar@{.}(2,0);(4,-10)
\ar@{.}(2,0);(6,-4)
\ar@{.}(2,0);(8,0)
\ar@{.}(8,0);(8,-4)
\ar@{.}(8,0);(14,-2)
\ar@{.}(0,-2);(14,-2)
\ar@{.}(0,-2);(6,-4)
\ar@{.}(0,-2);(0,-6)
\ar@{.}(0,-6);(0,-22)
\ar@{.}(0,-6);(2,-16)
\ar@{-}(0,-6);(4,-10)
\ar@{-}(0,-6);(6,-6)
\ar@{.}(6,-4);(8,-4)
\ar@{.}(8,-4);(14,-4)
\ar@{-}(4,-10);(6,-10)
\ar@{.}(2,-16);(4,-16)
\ar@{.}(0,-22);(2,-22)
\ar@{.}(14,-2);(14,-4)
\ar@{.}(6,-4);(6,-6)
\ar@{-}(6,-6);(6,-10)
\ar@{.}(4,-10);(4,-12)
\ar@{.}(4,-12);(4,-16)
\ar@{.}(2,-16);(2,-18)
\ar@{.}(2,-18);(2,-22)
\endxy
\end{array}
\end{array}
\]
\[
\begin{array}{cccc}
\begin{array}{c}
\def\objectstyle{\scriptscriptstyle}
\xy
\POS(0,0)*+{0},(2,0)*+{1},(8,0)*+{2},(0,-2)*+{2},(0,-6)*+{1},(0,-22)*+{0},(2,-22)*+{1},(2,-16)*+{0},(2,-18)*+{2},(4,-16)*+{1},(4,-10)*+{0},(4,-12)*+{2},(6,-10)*+{1},(6,-4)*+{0},(6,-6)*+{2},(8,-4)*+{1},(14,-2)*+{0},(14,-4)*+{2},
\ar@{-}(0,0);(2,0)
\ar@{-}(0,0);(0,-2)
\ar@{.}(2,0);(2,-16)
\ar@{.}(2,0);(4,-10)
\ar@{-}(2,0);(6,-4)
\ar@{.}(2,0);(8,0)
\ar@{.}(8,0);(8,-4)
\ar@{.}(8,0);(14,-2)
\ar@{.}(0,-2);(14,-2)
\ar@{-}(0,-2);(6,-4)
\ar@{.}(0,-2);(0,-6)
\ar@{.}(0,-6);(0,-22)
\ar@{.}(0,-6);(2,-16)
\ar@{.}(0,-6);(4,-10)
\ar@{.}(0,-6);(6,-6)
\ar@{.}(6,-4);(8,-4)
\ar@{.}(8,-4);(14,-4)
\ar@{.}(4,-10);(6,-10)
\ar@{.}(2,-16);(4,-16)
\ar@{.}(0,-22);(2,-22)
\ar@{.}(14,-2);(14,-4)
\ar@{.}(6,-4);(6,-6)
\ar@{.}(6,-6);(6,-10)
\ar@{.}(4,-10);(4,-12)
\ar@{.}(4,-12);(4,-16)
\ar@{.}(2,-16);(2,-18)
\ar@{.}(2,-18);(2,-22)
\endxy
\end{array}
&
\begin{array}{c}
\def\objectstyle{\scriptscriptstyle}
\xy
\POS(0,0)*+{0},(2,0)*+{1},(8,0)*+{2},(0,-2)*+{2},(0,-6)*+{1},(0,-22)*+{0},(2,-22)*+{1},(2,-16)*+{0},(2,-18)*+{2},(4,-16)*+{1},(4,-10)*+{0},(4,-12)*+{2},(6,-10)*+{1},(6,-4)*+{0},(6,-6)*+{2},(8,-4)*+{1},(14,-2)*+{0},(14,-4)*+{2},
\ar@{.}(0,0);(2,0)
\ar@{.}(0,0);(0,-2)
\ar@{.}(2,0);(2,-16)
\ar@{.}(2,0);(4,-10)
\ar@{.}(2,0);(6,-4)
\ar@{.}(2,0);(8,0)
\ar@{.}(8,0);(8,-4)
\ar@{.}(8,0);(14,-2)
\ar@{.}(0,-2);(14,-2)
\ar@{-}(0,-2);(6,-4)
\ar@{-}(0,-2);(0,-6)
\ar@{.}(0,-6);(0,-22)
\ar@{.}(0,-6);(2,-16)
\ar@{.}(0,-6);(4,-10)
\ar@{-}(0,-6);(6,-6)
\ar@{.}(6,-4);(8,-4)
\ar@{.}(8,-4);(14,-4)
\ar@{.}(4,-10);(6,-10)
\ar@{.}(2,-16);(4,-16)
\ar@{.}(0,-22);(2,-22)
\ar@{.}(14,-2);(14,-4)
\ar@{-}(6,-4);(6,-6)
\ar@{.}(6,-6);(6,-10)
\ar@{.}(4,-10);(4,-12)
\ar@{.}(4,-12);(4,-16)
\ar@{.}(2,-16);(2,-18)
\ar@{.}(2,-18);(2,-22)
\endxy
\end{array}
&
\begin{array}{c}
\def\objectstyle{\scriptscriptstyle}
\xy
\POS(0,0)*+{0},(2,0)*+{1},(8,0)*+{2},(0,-2)*+{2},(0,-6)*+{1},(0,-22)*+{0},(2,-22)*+{1},(2,-16)*+{0},(2,-18)*+{2},(4,-16)*+{1},(4,-10)*+{0},(4,-12)*+{2},(6,-10)*+{1},(6,-4)*+{0},(6,-6)*+{2},(8,-4)*+{1},(14,-2)*+{0},(14,-4)*+{2},
\ar@{-}(0,0);(2,0)
\ar@{-}(0,0);(0,-2)
\ar@{.}(2,0);(2,-16)
\ar@{.}(2,0);(4,-10)
\ar@{.}(2,0);(6,-4)
\ar@{-}(2,0);(8,0)
\ar@{.}(8,0);(8,-4)
\ar@{-}(8,0);(14,-2)
\ar@{-}(0,-2);(14,-2)
\ar@{.}(0,-2);(6,-4)
\ar@{.}(0,-2);(0,-6)
\ar@{.}(0,-6);(0,-22)
\ar@{.}(0,-6);(2,-16)
\ar@{.}(0,-6);(4,-10)
\ar@{.}(0,-6);(6,-6)
\ar@{.}(6,-4);(8,-4)
\ar@{.}(8,-4);(14,-4)
\ar@{.}(4,-10);(6,-10)
\ar@{.}(2,-16);(4,-16)
\ar@{.}(0,-22);(2,-22)
\ar@{.}(14,-2);(14,-4)
\ar@{.}(6,-4);(6,-6)
\ar@{.}(6,-6);(6,-10)
\ar@{.}(4,-10);(4,-12)
\ar@{.}(4,-12);(4,-16)
\ar@{.}(2,-16);(2,-18)
\ar@{.}(2,-18);(2,-22)
\endxy
\end{array}
&\begin{array}{c}
\def\objectstyle{\scriptscriptstyle}
\xy
\POS(0,0)*+{0},(2,0)*+{1},(8,0)*+{2},(0,-2)*+{2},(0,-6)*+{1},(0,-22)*+{0},(2,-22)*+{1},(2,-16)*+{0},(2,-18)*+{2},(4,-16)*+{1},(4,-10)*+{0},(4,-12)*+{2},(6,-10)*+{1},(6,-4)*+{0},(6,-6)*+{2},(8,-4)*+{1},(14,-2)*+{0},(14,-4)*+{2},
\ar@{.}(0,0);(2,0)
\ar@{.}(0,0);(0,-2)
\ar@{.}(2,0);(2,-16)
\ar@{.}(2,0);(4,-10)
\ar@{.}(2,0);(6,-4)
\ar@{.}(2,0);(8,0)
\ar@{.}(8,0);(8,-4)
\ar@{.}(8,0);(14,-2)
\ar@{-}(0,-2);(14,-2)
\ar@{-}(0,-2);(6,-4)
\ar@{.}(0,-2);(0,-6)
\ar@{.}(0,-6);(0,-22)
\ar@{.}(0,-6);(2,-16)
\ar@{.}(0,-6);(4,-10)
\ar@{.}(0,-6);(6,-6)
\ar@{-}(6,-4);(8,-4)
\ar@{-}(8,-4);(14,-4)
\ar@{.}(4,-10);(6,-10)
\ar@{.}(2,-16);(4,-16)
\ar@{.}(0,-22);(2,-22)
\ar@{-}(14,-2);(14,-4)
\ar@{.}(6,-4);(6,-6)
\ar@{.}(6,-6);(6,-10)
\ar@{.}(4,-10);(4,-12)
\ar@{.}(4,-12);(4,-16)
\ar@{.}(2,-16);(2,-18)
\ar@{.}(2,-18);(2,-22)
\endxy
\end{array}
\end{array}
\] 
The first five diagrams are the five Step 1 relations, the
last four the Step 2 relations.}
\end{example}
\begin{lemma}\label{vertex_0}
For any double index $(z_1,z_2)$ either there is precisely one
path out of vertex $0$ with that double index, or there are none.
\end{lemma}
\begin{proof}
By the above we just need to prove that all the basic squares in
$F$
out of the left rail commute. This is just a bookkeeping exercise:\\
\emph{Case 1: $n=1$}.  This is an easy extension of the $A_{4,1}$
example above.\\
\emph{Case 2: $n>1$}. We work through the basic squares leaving
$1$ (which we'll see, together with some basic squares leaving
$0$, correspond to the Step 1 relations) and then work upwards: if
$\alpha_1=2$, then the only basic square leaving 1 is
$\cl{1}{0}\an{0}{1}=\an{1}{2}\cl{2}{1}$, so we may assume
$\alpha_1>2$.  Then, as in Example~\ref{basic_squares_A11,3}, we
get $k_s\an{0}{1}=k_{s+1}\CL{0}{1}$  and above it
$\an{0}{1}k_s=\CL{0}{1}k_{s+1}$ for all $0\leq s< u_1$, and then end
with $k_{u_1}\an{0}{1}=\an{1}{2}\cl{2}{1}$.  Thus all basic
squares
leaving $1$ (and the corresponding ones leaving $0$) on the left rail commute.\\
Now for basic squares leaving $t$ on the left rail with $1<t<n$
(if such $t$ exist): if $\alpha_t=2$, then the only basic square is
$\cl{t}{t-1}\an{t-1}{t}=\an{t}{t+1}\cl{t+1}{t}$, so we may assume
that $\alpha_t>2$.  Certainly we have the basic square
$\cl{t}{t-1}\an{t-1}{t}=k_{v_t}\CL{0}{t}$ and above it
$\CL{0}{t}k_{v_t}=\AN{0}{l_{{V_t}}}k_{V_t}$.  If $\alpha_t>3$ we
also have the basic squares $k_{s}\AN{0}{t}=k_{s+1}\CL{0}{t}$ and
above it $\AN{0}{t}k_s=\CL{0}{t}k_{s+1}$ for all $v_t\leq s< u_t$.
The final basic square out of $t$ is
$k_{u_t}\AN{0}{t}=\an{t}{t+1}\cl{t+1}{t}$.\\
For the basic squares leaving $n$ on the left rail: if
$\alpha_n=2$, then the only basic square is
$\cl{n}{n-1}\an{n-1}{n}=\an{n}{0}\cl{0}{n}$ and above it
$\cl{0}{n}\an{n}{0}=\AN{0}{l_{{V_n}}}k_{V_n}$.  Hence assume
$\alpha_n>2$.  Then $\cl{n}{n-1}\an{n-1}{n}=k_{v_n}\cl{0}{n}$ and
above it $\cl{0}{n}k_{v_n}=\AN{0}{l_{{V_n}}}k_{V_n}$ .  The only
basic squares remaining are $k_{s}\AN{0}{n}=k_{s+1}\cl{0}{n}$ and
above it $\AN{0}{n}k_{s}=\cl{0}{n}k_{s+1}$ for all $v_n\leq s<
u_n$ (recall $k_{u_n}=\an{n}{0}$).
\end{proof}

\begin{cor}\label{map_inj}
For any double index $(z_1,z_2)$ and any vertex $t$, either there
is precisely one path out of vertex $t$ with that double index, or
there are none.
\end{cor}
\begin{proof}
 To obtain the web of paths of vertex $n$, delete the top row in the
web of paths of vertex $0$ and decrease the first index in every
double index by 1.  All squares in this web of paths commute
because they commute in the web of paths for vertex $0$.  For vertex
$n-1$ delete all the rows above the $n-1$ on the left rail and
decrease the double indices accordingly. Again all squares in this
web of paths commute since they commute in the web of paths for vertex
0.  Continue in this fashion.
\end{proof}

We now reach the main theorem which shows that the
algebraically-constructed ring (the endomorphism ring of the
specials) is isomorphic to the geometrically-constructed ring (the
reconstruction algebra).  For a third construction of the same
non-commutative ring, see Section 5.
\begin{thm}\label{endo specials iso reconstruct}
For $G=\frac{1}{r}(1,a)$, let $T_{r,a}=\oplus_{p=1}^{n+1}S_{i_p}$.
Then \[ A_{r,a}\cong \t{End}_{\C{}[x,y]^G}(T_{r,a}).\]
\end{thm}
\begin{proof}
We already have a map $\phi:\C{}Q\rightarrow \t{End}(\oplus_{p=1}^{n+1}S_{i_p})$, which by Proposition~\ref{map_onto} is surjective.  It is straightforward to see that all the reconstruction relations are sent to zero and so belong to the kernel, since the double index of any relation corresponds to the
double index $(z_1,z_2)$ of the cycle $x^{z_1}y^{z_2}$ in the
endomorphism ring that it represents.  The content in the theorem is that there are no more
relations.  But this is just Corollary~\ref{map_inj}.
\end{proof}

\begin{remark}
\t{If $a=r-1$, then the group $\frac{1}{r}(1,r-1)$ is inside
$SL(2,\C{})$, all CM modules are special and
$T_{r,r-1}=\C{}[x,y]$, so this theorem reduces to the well known
\[
A_{r,r-1}=\t{pre-projective algebra}\cong
\t{End}_{\C{}[x,y]^G}(\C{}[x,y])\cong \C{}[x,y]\# G.
\]}
\end{remark}

\begin{cor}\label{centrefg}
The centre of $A_{r,a}$ is $\C{}[x,y]^{\frac{1}{r}(1,a)}$, and furthermore $A_{r,a}$ is a finitely generated module over its centre.  Consequently $A_{r,a}$ is a noetherian PI ring.
\end{cor}
\begin{proof}
Denote $R=\C{}[x,y]^{\frac{1}{r}(1,a)}$.  Then since CM modules are torsion-free there is an embedding of $R$ into $\t{End}_{R}(T_{r,a})$ which clearly has image inside the centre.  It is the whole of the centre since $\t{Hom}_{R}(S_{i_p},S_{i_p})\cong R$ for all $p$.  The fact that $A_{r,a}$ is finitely generated as an $R$-module is immediate from its description as an endomorphism ring.  The rest is standard.
\end{proof}

\section{Moduli Space of Representations and the Minimal Resolution}
The minimal resolutions of cyclic quotient singularities are well
understood by a construction of Fujiki (see for example
\cite[2.7]{Wunram_cyclicBook}).  More recently there is an easier
algorithm using toric geometry techniques \cite{Reid_cyclic} which
coincides with the $G\t{-Hilb}$ description (\cite{Kidoh},
\cite{Ishii}).

In this section we prove that for any group $G=\frac{1}{r}(1,a)$
we can obtain the minimal resolution of the singularity $\C{2}/G$
from the moduli space of the reconstruction algebra $A_{r,a}$,
thus giving yet another description of the minimal resolution.
This may not be entirely surprising (by construction!), but it is
important since by Theorem~\ref{endo specials iso reconstruct} we
could have defined $A_{r,a}$ without prior knowledge of the
minimal resolution.

For a summary of moduli space techniques we refer the reader to
\cite{King_moduli}, \cite{King_tilting}.  For $G=\frac{1}{r}(1,1)$
(i.e.\ the reconstruction algebra with the $n=1$ relations)
everything is trivial, so we assume $n\geq 2$. With respect to the
ordering of the vertices as in Section 2, fix for the rest of this
paper the dimension vector $\alpha=(1,1,\hdots,1)$ and fix the
generic stability condition $\theta=(-n,1,\hdots,1)$. The point is
that when considering representations of this dimension vector the
maps are just scalars so the relations reduce in complexity.  As
we shall see, the stability condition is chosen to be `blind' to
the arrows $k_1,\hdots k_{\sum(\alpha_1-2)}$, and so we have an open covering of the moduli space by the same number of open sets as in the preprojective case (i.e.\ $n+1$ open sets).  It is the relations that
ensure each of the open sets is still $\C{2}$, and standard geometric
arguments give minimality. 

\begin{defin}
For $0\leq t\leq n$ define the open set $W_t$ in $\xymatrix{\t{Rep}(A_{r,a},\alpha){/\!\!/}_{\theta} \t{GL}}$ as follows: $W_0$ is defined by the condition $\CL{0}{1}\neq 0$, $W_n$ by the condition $\AN{0}{n}\neq 0$, and for $1\leq t\leq n-1$, $W_t$ is defined by the conditions $\CL{0}{t+1}\neq 0$ and $\AN{0}{t}\neq 0$.
\end{defin}

\begin{remark}
\t{For toric geometers, below is the correspondence between the above open sets and the standard toric charts on the minimal resolution.  We also state for reference the result of Lemma~\ref{opencover}, which gives the position of where (if we change basis so that the specified non-zero arrows in the definition of the open sets are actually the identity) the co-ordinates can be read off the quiver.}
\[
\begin{array}{rclc}
W_0&  \longleftrightarrow & \t{Spec }\C{}\left[x^r,\frac{y}{x^a}  \right] & (\cl{1}{0},\an{0}{1})\\
&  \vdots &  &\\
\begin{array}{c}W_t\\ \mbox{for }1\leq t\leq n-1
\end{array}& \longleftrightarrow
& \t{Spec }\C{}\left[
\frac{x^{i_t}}{y^{j_t}},\frac{y^{j_{t+1}}}{x^{i_{t+1}}}\right] & (\cl{t+1}{t},\an{t}{t+1})\\
&  \vdots &  &\\
W_n&  \longleftrightarrow &
\t{Spec }\C{}\left[\frac{x}{y^{j_n}},y^r \right] &
(\cl{0}{n},\an{n}{0}).
\end{array}
\]
\end{remark}
\begin{lemma}
For the group $G=\frac{1}{r}(1,a)$ with notation as above,
$\{W_t: 0\leq t\leq n\}$ forms an open cover of the moduli space
$\xymatrix{\t{Rep}(A_{r,a},\alpha){/\!\!/}_{\theta} \t{GL}}$.
\end{lemma}
\begin{proof} Suppose $M\in\t{Rep}(A_{r,a},\alpha)$ is $\theta$-stable.   It is clear from the stability condition that we need, for every vertex $i\neq 0$, a non-zero path from vertex $0$ to vertex $i$.  Now if $\an{0}{1}=0$, then to get to vertex $1$ clearly we need $\CL{0}{1}\neq 0$, and so $M$ is in $W_0$. Otherwise we can assume
$\an{0}{1}\neq 0$.  If $\an{1}{2}=0$, then to get to vertex 2 we need  $\CL{0}{2}\neq 0$, and so $M$ is in $W_1$.  Continuing in this fashion, it is obvious that $\{W_t: 0\leq t\leq n\}$ forms an open cover of the moduli space.
\end{proof}
\begin{lemma}\label{opencover}
\t{(i)} Each representation in $W_0$ is determined by
$(\cl{1}{0},\an{0}{1})\in \C{2}$.\\\t{(ii)} For every $1\leq t\leq
n-1$, each representation in $W_t$ is determined by
$(\cl{t+1}{t},\an{t}{t+1})\in \C{2}$.\\\t{(iii)} Each
representation in $W_n$ is determined by $(\an{n}{0},\cl{0}{n})\in
\C{2}$.\\ Thus every open set in the cover is just $\C{2}$.
\end{lemma}
\begin{proof}
(i) We can set $\cl{0}{n}=\cl{n}{n-1}=\hdots=\cl{2}{1}=1$.  We
proceed anticlockwise round the vertices of the quiver (starting
at the 0th vertex), showing at each stage that all arrows leaving
the
vertex are determined by $\cl{1}{0}$ and $\an{0}{1}$.\\
\textit{Vertex 0:}
Trivial, as the only arrows leaving are $\cl{0}{n}=1$ and $\an{0}{1}$.\\
\textit{Vertex 1:} If $\alpha_1=2$, then the only two arrows
leaving are $\an{1}{2}$ and $\cl{1}{0}$.  The Step 1 relations
give $\an{1}{2}=\cl{1}{0}\an{0}{1}$.  Thus we may assume that
$\alpha_1>2$, so we have
$\cl{1}{0}=k_0,k_1,\hdots,k_{u_1},\an{1}{2}$ leaving the vertex.
But now the Step 1 relations give
\begin{eqnarray*}
\cl{1}{0}\an{0}{1}&=&k_1\\ k_1\an{0}{1}&=&k_2\\ &\vdots& \\
k_{u_1-1}\an{0}{1}&=&k_{u_1}\\ k_{u_i}\an{0}{1}&=&\an{1}{2},
\end{eqnarray*}
so it is clear that $k_1,\hdots,k_{u_1},\an{1}{2}$ can be expressed in terms of
$\cl{1}{0}$ and $\an{0}{1}$.\\ 
\textit{Vertex $s$ for $1<s<n$:} If
$\alpha_s=2$, then only arrows leaving are $\cl{s}{s-1}=1$ and
$\an{s}{s+1}$.  The Step $s$ relations give
$\an{s}{s+1}=\an{s-1}{s}$, and by work at previous vertices we know
that $\an{s-1}{s}$ is determined by $\cl{1}{0}$ and $\an{0}{1}$;
hence so is $\an{s+1}{s}$.  Thus we may assume $\alpha_s>2$, and so the arrows leaving vertex $s$ are
$k_{v_s},\hdots,k_{u_s},\cl{s}{s-1}=1,\an{s}{s+1}$.  But by the
Step $s$ relations we know
\begin{eqnarray*}
k_{v_s}&=&\an{s-1}{s}\\ k_{v_s+1}&=& k_{v_s}\AN{0}{s} \\ &\vdots& \\
k_{u_s}&=&k_{u_s-1}\AN{0}{s}\\ \an{s}{s+1}&=&k_{u_s}\AN{0}{s},
\end{eqnarray*}
which, by work at the previous vertices, can all be expressed in
terms of $\cl{1}{0}$ and $\an{0}{1}$.\\
\textit{Vertex $n$:} If $\alpha_n=2$, then again everything is
trivial, and so we may assume $\alpha_n>2$, in which case the arrows
$k_{v_n},\hdots,k_{u_n}=\an{n}{0},\cl{n}{n-1}=1$ leave vertex $n$.
The step $n$ relations give
\begin{eqnarray*}
k_{v_n}&=&\an{n-1}{n}\\ k_{v_n+1}&=& k_{v_n}\AN{0}{n} \\ &\vdots& \\
k_{u_n}&=&k_{u_s-1}\AN{0}{s},
\end{eqnarray*}
which again by work at the other vertices can be expressed in
terms of $\cl{1}{0}$ and $\an{0}{1}$.\\
(iii) Follows immediately by Lemma~\ref{read_backwards}.\\
(ii) We can set
$\cl{0}{n}=\hdots=\cl{t+2}{\,t+1}=1=\an{0}{1}=\hdots=\an{t-1}{t}$.
As above we show that the arrows leaving every vertex are
determined by $\cl{t+1}{t}$ and $\an{t}{t+1}$, but we now work
anticlockwise from vertex $t+1$ to vertex $0$ and then work clockwise
from vertex $t$ to vertex $1$: \\ \textit{Vertex $t+1$:} If
$\alpha_{t+1}=2$, then the only arrows leaving are $\cl{t+1}{t}$
and $\an{t+1}{t+2}$.  The relations give
$\an{t+1}{t+2}=\cl{t+1}{t}\an{t}{t+1}$.  Hence we may assume
$\alpha_{t+1}>2$, and so the arrows leaving vertex $t+1$ are
$k_{v_{t+1}},\hdots,k_{u_{t+1}},\cl{t+1}{t},\an{t+1}{t+2}$.  The
Step $t+1$ relations give
\begin{eqnarray*}
k_{v_{t+1}}&=&\cl{t+1}{t}\an{t}{t+1}\\ k_{v_{t+1}+1}&=& k_{v_{t+1}}\AN{0}{t+1}= k_{v_{t+1}}\an{t}{t+1} \\ &\vdots& \\
k_{u_{t+1}}&=&k_{u_{t+1}-1}\AN{0}{t+1}=k_{u_{t+1}-1}\an{t}{t+1}\\
\an{t+1}{t+2}&=&k_{u_{t+1}}\AN{0}{t+1}=k_{u_{t+1}}\an{t}{t+1},
\end{eqnarray*}
which therefore can all be expressed in terms of $\cl{t+1}{t}$ and
$\an{t}{t+1}$.\\ \textit{Vertex $s$ for $n<s<t+1$:} If
$\alpha_{s}=2$, then the only arrows leaving are $\cl{s}{s-1}=1$
and $\an{s}{s+1}$.  The relation gives $\an{s}{s+1}=\an{s-1}{s}$,
and by work at previous vertices we know that $\an{s-1}{s}$ is
determined by $\cl{t+1}{t}$ and $\an{t}{t+1}$; hence so is
$\an{s}{s+1}$.  Hence assume $\alpha_s>2$, and so the arrows
leaving are $k_{v_{s}},\hdots,k_{u_{s}},\cl{s}{s-1},\an{s}{s+1}$.
The Step $s$ relations give
\begin{eqnarray*}
k_{v_{s}}&=&\an{s-1}{s}\\ k_{v_{s}+1}&=& k_{v_{s}}\AN{0}{s}= k_{v_{s}}\AN{t}{s} \\ &\vdots& \\
k_{u_{s}}&=&k_{u_{s}-1}\AN{0}{s}=k_{u_{s}-1}\AN{t}{s}\\
\an{s}{s+1}&=&k_{u_{s}}\AN{0}{s}=k_{u_{s}}\AN{t}{s},
\end{eqnarray*}
which by work at the previous vertices can all be expressed in
terms of $\cl{t+1}{t}$ and $\an{t}{t+1}$.\\
\textit{Vertex $n$:} Similar to the above case.\\
\textit{Vertex 0:} Only arrows leaving are $\cl{0}{n}$ and
$\an{0}{1}$, both of which are 1.\\
We now start at vertex $t$ and work clockwise:\\
\textit{Vertex $t$:} If $\alpha_t=2$, then the only arrows leaving
are $\cl{t}{t-1}$ and $\an{t}{t+1}$; the relations give
$\cl{t}{t-1}=\an{t}{t+1}\cl{t+1}{t}$.  Hence assume $\alpha_t>2$,
and so the arrows leaving are
$k_{v_{t}},\hdots,k_{u_{t}},\cl{t}{t-1},\an{t}{t+1}$.  The Step
$t$ relations (read backwards) give
\begin{eqnarray*}
k_{u_{t}}&=&\an{t}{t+1}\cl{t+1}{t}\\ k_{u_{t}-1}&=& k_{u_{t}}\CL{0}{t}= k_{u_{t}}\cl{t+1}{t} \\ &\vdots& \\
k_{v_{t}}&=&k_{v_{t}+1}\CL{0}{t}=k_{v_{t}+1}\cl{t+1}{t}\\
\cl{t}{t-1}&=&k_{v_{t}}\CL{0}{t}=k_{v_{t}}\cl{t}{t-1},
\end{eqnarray*}
which therefore can all be expressed in terms of $\cl{t+1}{t}$ and
$\an{t}{t+1}$.\\ \textit{Vertex $s$ for $1\leq s<t$:} Similar to
the above; read the Step $s$ relations backwards and use work at
the previous vertices.
\end{proof}
\begin{thm}\label{modulimain}
Keeping $\alpha$ and $\theta$ fixed from before,
\[
\xymatrix{{\t{Rep}(A_{r,a},\alpha){/\!\!/}_{\theta} \t{GL}}\ar[r] &
{\C{2}/\frac{1}{r}(1,a)}}
\]
is the minimal resolution of singularities.
\end{thm}
\begin{proof}
First note that $W_{t-1}$ and $W_{t}$ glue together to give
$\s{O}_{\P}(-\alpha_t)$ for each $1\leq t\leq n$.  To see this let $(a,b)\in W_{t-1}$ with $b\neq 0$.  Then simply changing basis at vertex $t$ by dividing all arrows into vertex $t$ by $b$ and multiplying all arrows out by $b$ gives us
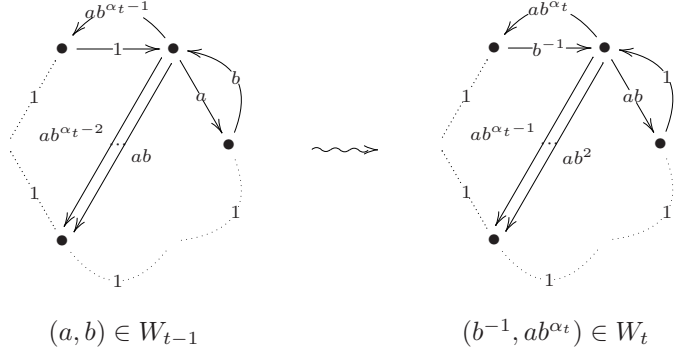
\[
\begin{array}{ccc}
\begin{array}{c}
\def\alphanum{\ifcase\xypolynode\or \bullet\or \bullet\or \bullet\or \or \bullet
\or \fi}
\xy \xygraph{!{/r3.5pc/:} 
[] !P6"A"{~>{} ~*{\alphanum}}
}
\ar|{a} "A2";"A1"
\ar|{1} "A3";"A2"
\ar@{.}|1 "A4";"A3"
\ar@{.}|1 "A5";"A4"
\ar@/_1.25pc/|{b} "A1";"A2"
\ar@/_1.25pc/|{ab^{\alpha_t-1}} "A2";"A3"
\ar@{.}@/_1.25pc/|1 "A5";"A6"
\ar@{.}@/_1.25pc/|1 "A6";"A1"
\ar_{ab^{\alpha_t-2}}@<-0.5ex> "A2";"A5"
\ar@{}|{\cdots} "A2";"A5"
\ar^{ab}@<0.5ex> "A2";"A5"
\endxy
\end{array}
&
\begin{array}{c}
\xy
\POS (0,0)="0", (10,0)="1"
\POS"0"\ar@{~>}"1"
\endxy
\end{array}
&
\begin{array}{c}
\def\alphanum{\ifcase\xypolynode\or \bullet\or \bullet\or \bullet\or \or \bullet
\or \fi}
\xy \xygraph{!{/r3.5pc/:} 
[] !P6"A"{~>{} ~*{\alphanum}}
}
\ar|{ab} "A2";"A1"
\ar|{b^{-1}} "A3";"A2"
\ar@{.}|1 "A4";"A3"
\ar@{.}|1 "A5";"A4"
\ar@/_1.25pc/|{1} "A1";"A2"
\ar@/_1.25pc/|{ab^{\alpha_t}} "A2";"A3"
\ar@{.}@/_1.25pc/|1 "A5";"A6"
\ar@{.}@/_1.25pc/|1 "A6";"A1"
\ar_{ab^{\alpha_t-1}}@<-0.5ex> "A2";"A5"
\ar@{}|{\cdots} "A2";"A5"
\ar^{ab^2}@<0.5ex> "A2";"A5"
\endxy
\end{array}\\
\\
(a,b)\in W_{t-1}&&(b^{-1},ab^{\alpha_t})\in W_t
\end{array}
\]
Consequently above the singularity there is a string of $\P$'s, each with
self-intersection number $\leq -2$. None of these can be
contracted to give a smaller resolution.
\end{proof}
\begin{remark}\t{For finite subgroups of $SL(2,\C{})$ the above theorem remains true if we replace the fixed $\theta$ by an arbitrary generic stability condition \cite{Cassens_Slodowy}.  However, in the $GL(2,\C{})$ case if we choose a different stability condition it is \emph{not} true in general that the above theorem holds, since the moduli may have components.  A concrete example is $\frac{1}{3}(1,1)$.  Thus in the non-Gorenstein case the question of stability is much more subtle.}
\end{remark}

\section{Tilting Bundles}
We want to show that the minimal resolution $\w{X}$ of the
singularity $\C{2}/\frac{1}{r}(1,a)$ is derived equivalent to the
reconstruction algebra $A_{r,a}$.  To do this, we search for a
tilting bundle.  During the production of this paper this result
has been independently proved by Craw
\cite{Craw_independant_moduli}, who points out that it actually
follows immediately from a result of Van den Bergh \cite[Thm
B]{vdb_flops_paper}.  The proof here uses a simple trick involving
an ample line bundle.

\begin{defin}
Suppose $\T$ is a triangulated category with small direct sums. An
object $C\in\T$ is called compact if for any set of objects $X_i$,
the natural map \[ \coprod \t{Hom}(C,X_i)\rightarrow
\t{Hom}\left(C,\coprod X_i\right)\] is an isomorphism.
\end{defin}
Denote by $\loc{\X}$ the smallest full triangulated subcategory of
$\T$ closed under infinite sums containing $\X$.  Note this is
necessarily closed under direct summands.
\begin{defin}
Let $\T$ be a triangulated category with small direct sums.  We
say $\T$ is compactly generated if there exists a set of compact
objects $\X$ such that $\loc{\X}=\T$.
\end{defin}
\begin{defin}
A vector bundle $\s{E}$ of finite rank is called a tilting bundle
if\\ \t{(1)} $\t{Ext}^i(\s{E},\s{E})=0$ for all $i>0$, \\ \t{(2)}
$\loc{\s{E}}=\dqcoh{X}$.
\end{defin}

It is standard that if $\w{X}$ has a tilting bundle $\s{E}$, then $\w{X}$ and $\t{End}_{\w{X}}(\s{E})$ are derived equivalent (see Theorem~\ref{derived_equiv} below).  In our situation the bundle to consider comes from Wunram's geometric description of the special CM modules as full sheaves: 

\begin{defin}\cite{Esnault_full}
A sheaf $\s{F}$ on $\w{X}$ is called full if\\
\t{(1)} $\s{F}$ is locally free,\\
\t{(2)} $\s{F}$ is generated by global sections,\\
\t{(3)} $H^1(\w{X},\s{F}^\vee\otimes \omega_{\w{X}})=0$,
where $\omega_{\w{X}}$ is the canonical module.
\end{defin}
Denoting the minimal resolution by
$\w{X}\stackrel{\pi}\longrightarrow {\C{2}/G}$, then given any CM module $M$, it is true that
\[
\w{M}:=\pi^\ast M/\t{torsion}
\]
is a full sheaf.  In fact, full sheaves are in 1-1 correspondence
with indecomposable CM modules by work of Esnault
\cite{Esnault_full}; the inverse map is global sections.  Denote the functor $\t{Hom}(-,R)$ by ${}^\ast$ and note that if $M$ is any CM module, then $M^\ast=\pi_\ast ((\w{M})^\vee)$. 

The definition of a special CM module was originally
stated in terms of the corresponding full sheaf:
\begin{lemma}\cite{Wunram_generalpaper}\label{special_full_characterisation}
$S_{t}$ is a special CM module $\iff$
$H^1(\w{S_{t}}^\vee)=0$.
\end{lemma}

To obtain a derived equivalence between the minimal resolution
$\w{X}$ and the reconstruction algebra $A_{r,a}$, we shall
show that $\s{E}=\oplus_{p=1}^{n+1}\w{S_{i_p}}$ is a tilting bundle with endomorphism ring
$A_{r,a}$. First we prove the statement on the endomorphism ring: note that the following lemma shows that the three ways to produce a non-commutative ring all give the same answer.
\begin{lemma}\label{endo_of_fullspecials_is_Ara}
$\t{End}_{\w{X}}\left(\s{E}\right)=\t{End}_{\w{X}}\left(\oplus_{p=1}^{n+1}
\w{S_{i_p}}\right)\cong
\t{End}_{\C{}[x,y]^G}\left(\oplus_{p=1}^{n+1} S_{i_p}\right)\cong
A_{r,a}.$
\end{lemma}
\begin{proof}
The last isomorphism is Theorem~\ref{endo specials iso
reconstruct}. The first isomorphism follows from the statement
\[
\t{Hom}_{\w{X}}\left(\w{S_{i_p}},\w{S_{i_q}}\right)\cong \t{Hom}_{\C{}[x,y]^G}(S_{i_p},S_{i_q}),
\]
which is true since both are reflexive and isomorphic away from the singular point (see e.g. \cite[3.1]{Wemyss_GL2}).
\end{proof}

Now for every pair $p,q$ with $1\leq p,q\leq n$, by choosing a generic section of ${\w{S_{i_q}}} \oplus
{\w{S_{i_p}}}$ we have a short exact sequence
\[
\xymatrix{0\ar[r]& {\s{O}}\ar[r] & {\w{S_{i_q}}} \oplus
{\w{S_{i_p}}}\ar[r] & {\w{S_{i_q}}}\otimes{\w{S_{i_p}}} \ar[r] &0}.
\]
(This can also be constructed locally, using the open cover in Theorem~\ref{modulimain}.)  Tensoring the above by $\w{S_{i_p}}^\vee$ gives the exact sequence
\[
\xymatrix{0\ar[r]& {\w{S_{i_p}}}^\vee\ar[r] &
{\left(\w{S_{i_p}}^\vee\otimes \w{S_{i_q}}\right)} \oplus
{\s{O}}\ar[r] & {\w{S_{i_q}}}\ar[r] &0}.\tag{$B_{p,q}$}
\]
\begin{lemma}\label{exts_ok}
$\t{Ext}^r(\s{E},\s{E})=0$ for all $r>0$.
\end{lemma}
\begin{proof}
Since the singularity is rational, $H^r(\s{O})=0$ for all $r>0$.
Further, $\w{S_{i_p}}$ is generated by global sections so
$H^1(\w{S_{i_p}})=0$.  Thus using the short exact sequence
($B_{p,p}$)
\[
\xymatrix{0\ar[r]& {\w{S_{i_p}}}^\vee\ar[r] & {\s{O}}^2\ar[r] &
{\w{S_{i_p}}}\ar[r] &0}
\]
and Lemma~\ref{special_full_characterisation}, it is true that
$H^r(\w{S_{i_p}})=H^r(\w{S_{i_p}}^\vee)=0$ for all $r>0$.

But using ($B_{p,q}$) together with these facts shows that
$H^r(\w{S_{i_p}}^\vee\otimes \w{S_{i_q}})=0$ for all $r>0$ and
$1\leq p,q\leq n$. Hence
\[
\t{Ext}^r(\s{E},\s{E})={\bigoplus\limits_{p,q=1}^{n+1}H^r(\w{S_{i_p}}^\vee\otimes
\w{S_{i_q}}})=0\,\, \t{for all}\,\, r>0.
\]
\end{proof}

\begin{thm}\label{derived_equiv}
Let $\w{X}$ be the minimal resolution of the singularity
$\C{2}/\frac{1}{r}(1,a)$, let $A_{r,a}$ be the corresponding
reconstruction algebra and let $\s{E}=\oplus_{p=1}^{n+1}\w{S_{i_p}}$. Then:\\
\t{(1)} $\mathbf{R}\t{Hom}(\s{E},-)$ induces an
equivalence between
$\dqcoh{\w{X}}$ and $\da{A_{r,a}}$\\
\t{(2)} This equivalence restricts to an equivalence between
$\dbqcoh{\w{X}}$ and $\dba{A_{r,a}}$\\
\t{(3)} This equivalence restricts to an equivalence between
$\dbcoh{\w{X}}$ and $\db{A_{r,a}}$\\
\t{(4)} Since $\w{X}$ is smooth, $A_{r,a}$ has finite global
dimension.
\end{thm}
\begin{proof}
By Lemma~\ref{endo_of_fullspecials_is_Ara} and Lemma~\ref{exts_ok}
we need only prove that $\loc{\s{E}}=\dqcoh{\w{X}}$.  Since the first Chern class of $\s{L}:=\w{S_{i_1}}\otimes\hdots\otimes\w{S_{i_n}}$ consists entirely of ones, $\s{L}$ is ample, and so it is true
by \cite[1.10]{Neeman_Grothendieck} that
$\dqcoh{X}=\loc{\s{L}^{-\otimes n}:n\in\mathbb{N}}$. Hence it
suffices to prove that $\loc{\s{E}}$ contains all negative tensors
of the ample line bundle $\s{L}$.  But using the short exact sequences
$B_{p,q}$ together with suitable tensors of them (which give
triangles), this is indeed true: by the sequence $B_{p,p}$ it
follows that $\loc{\s{E}}$ contains $\w{S_{i_p}}^\vee$.  Now after
tensoring $B_{p,p}$ by $\w{S_{i_p}}^\vee$ it follows that
$\loc{\s{E}}$ contains $\w{S_{i_p}}^{\otimes -2}$.  Continuing in
this fashion $\loc{\s{E}}$ contains $\w{S_{i_p}}^{\otimes -t}$ for
all $t\geq 0$ and all $i_p$.  Now by considering the sequence
$B_{p,q}$ tensored by $\w{S_{i_q}}^\vee$, it follows that
$\loc{\s{E}}$ contains $\w{S_{i_p}}^\vee\otimes \w{S_{i_q}}^\vee$.
Continuing in this manner a simple inductive argument shows that
$\loc{\s{E}}$ contains
$(\w{S_{i_1}}\otimes\hdots\otimes\w{S_{i_n}})^{\otimes -t}$ for
all $t\geq 0$.  The result is now standard (see e.g.
\cite[7.6]{vdb_tilting}).
\end{proof}
Denoting by $\thick{\s{E}}$ the smallest thick full triangulated
subcategory containing $\s{E}$, it is true by Neeman-Ravenel
\cite{Neeman_smashing} that $\thick{\s{E}}$ coincides with the
compact objects of $\dqcoh{\w{X}}$.  By
\cite[2.3]{Neeman_Grothendieck} these are precisely the perfect
complexes, which, since $\w{X}$ is smooth, are the whole of
$\dbcoh{\w{X}}$. Thus it is also true that
$\thick{\s{E}}=\dbcoh{\w{X}}$.

\section{Homological Considerations}
It is well known that the preprojective algebra of an extended Dynkin diagram is a homologically
homogeneous ring of global dimension 2. We observed in
Theorem~\ref{derived_equiv} that for general labels
$[\alpha_1,\hdots,\alpha_n]$ the reconstruction algebra of type
$A$ also has finite global dimension.  Thus it is natural to ask
its value and whether it is homologically homogeneous (i.e.\ all its simple modules have the same projective dimension).

We shall prove in this section that
\[
\t{gldim}A_{r,a}=\left\{\begin{array}{l} 2\,\,\t{if}\,\,a=r-1,\\
3 \,\,\t{else}.\end{array}\right.
\]
and so $A_{r,a}$ is homologically homogeneous only when $r=a-1$,
i.e.\ when $G=\frac{1}{r}(1,a)\leq SL(2,\C{})$.  Furthermore, we show that the projective resolutions of the simples in the non-Azumaya locus are determined by intersection theory.

We first show that the Azumaya locus of $A_{r,a}$ coincides with the smooth locus of its centre
$R=\C{}[x,y]^G$.  Such a problem has been considered in \cite{Marjory} in a slightly more general setting, although here we give a direct argument.  The reason we desire such a result is that we can then `ignore' the simples in the Azumaya locus (i.e. those simple $A$-modules whose $R$-annihilator lies in the Azumaya locus defined below), as they correspond to smooth
points, and so their projective dimensions are easily controlled.
\begin{defin}
$A=A_{r,a}$ is a noetherian ring module finite over its centre
$R=\C{}[x,y]^{\frac{1}{r}(1,a)}$.  Define
\begin{eqnarray*}
\s{A}_A&=&\{ \m\in \t{max}R : A_\m \,\,\t{ is Azumaya over } R_\m \},
\end{eqnarray*}
where $\t{max}R$ is the set of maximal ideals of $R$.  The set $\s{A}_A$ is called the Azumaya locus of $A$.
\end{defin}

\begin{lemma}
$A=A_{r,a}$ is a prime ring of PI degree $n+1$.
\end{lemma}
\begin{proof}
Since $R$ is a domain, $\bf{0}$ is a prime ideal.  Denote $F=R_{\bf{0}}$ ($R$ localized at the zero ideal) to be the quotient field of $R$.  Since CM modules are torsion-free and $A\cong \t{End}_{R}(\oplus_{p=1}^{n+1} {S_{i_p}})$, non-zero elements in $R$ are not zero-divisors in $A$, and so 
\[
A\subseteq A\otimes F=A_{\bf{0}} \cong \t{End}_{R_{\bf{0}}}(\oplus_{p=1}^{n+1} {S_{i_p}}_{\bf{0}})=\t{End}_{R_{\bf{0}}}(\oplus_{p=1}^{n+1} {R}_{\bf{0}})=\t{End}_F(F^{n+1}),
\]
since the CM modules $S_{i_p}$ are free of rank 1 away from the singular locus of $R$.  Thus $A\subseteq A\otimes F=A_{\bf{0}}$, with $A_{\bf{0}}$ a classical right quotient ring of $A$. Since $A_{\bf{0}}$ is simple, $A$ is necessarily prime \cite[6.17]{Goodearl_Warfield}.  Now $A$ is a prime PI ring, so by definition its PI-degree is equal to $\t{dim}_{F}(F \otimes (\oplus_{p=1}^{n+1} {S_{i_p}}))=\t{dim}_{F}(F^{n+1})=n+1$.
\end{proof}
Throughout this section we denote by $\m_0$ the unique singular point of $R$.
\begin{lemma}\label{az}
$\s{A}_A=\t{max}R\backslash \{ \m_0\}$.
\end{lemma}
\begin{proof}
By Theorem~\ref{endo specials iso reconstruct}, ${A}_\m\cong \t{End}_{R_\m}(\oplus_{p=1}^{n+1} {S_{i_p}}_\m)$ for all $\m\in\t{max}R$.  But CM modules are free on the smooth locus, so if $\m\neq\m_0$, then ${S_{i_p}}_\m\cong R_\m$ for all $p$.  Consequently, ${A}_\m\cong M_{n+1}(R_\m)$ for any $\m\neq\m_0$, where $M_{n+1}(R_\m)$ is the ring of $(n+1)$-square matrices over $R_\m$ and thus $A_\m$ is Azumaya over $R_\m$.   This proves that $\t{max}R\backslash \{ \m_0\}\subseteq \s{A}_A$.  Equality holds since $A$ is a prime affine $\C{}$-algebra, finitely generated over its centre with finite global dimension.  For such rings it is well known that the Azumaya locus and the singular locus are disjoint (see e.g. \cite[III.1.8]{Brown_Goodearl}).  Alternatively, just observe that the one dimensional simple corresponding to vertex 0 is a simple $A$-module annihilated by $\m_0$ and that this does not have maximal dimension $n+1$ (equal to the PI degree).
\end{proof}
\begin{cor}\label{gldim_az}
For all $\m\in\s{A}_A$, $\t{gldim}A_\m=2$.
\end{cor}
\begin{proof}
By the above lemma $\m\neq\m_0$ with $A_\m\cong M_{n+1}(R_\m)$.  Since global dimension passes over morita equivalence, we have $\t{gldim}A_\m=\t{gldim}R_\m=2$, where the last equality holds since $R$ is equi-codimensional \cite[13.4]{Eisenbud_AlgGeom}.
\end{proof}
The hard work in the global dimension proof comes from computing the
projective resolutions of the one dimensional simples corresponding
to the vertices of $A_{r,a}$.

Let $Q$ be the quiver of the reconstruction algebra.  Denote its vertices by $Q_0$ and its arrows by $Q_1$.  Denote the relations by $\c{R}=\{ R_t\}$ and note that they are all admissible (i.e.\ contain no path of length $\leq 1$) and basic (i.e.\ each $R_t$ is a linear combination of paths, each with common head and tail).  Denote by $D_j$ the one dimensional simple corresponding to vertex $j\in Q_0$.  As is standard, for any paths $p$ we denote $t(p)$ to be the tail of $p$ and $h(p)$ to be the head. Consider the following complex:
\[
{\bigoplus_{t(R_i)=j}^{} e_{h(R_i)}A}\stackrel{d_2}\longrightarrow {\bigoplus_{t(a)=j} e_{h(a)}A}\stackrel{d_1}\longrightarrow e_jA\to D_j\to 0,
\]
where the left hand sum is over all relations with tail $j$ and the right hand sum is over all arrows with tail $j$.  The maps $d_1$ and $d_2$ are given as
\[
\begin{array}{ccc}
d_2: (g_i)\mapsto (\sum_{i}\partial_a R_ig_i)_{a}&&d_1: (f_a)\mapsto \sum_{a} a f_a
\end{array}
\]
respectively, where for any arrow $a$ and any path $p$ we define 
\[
\partial_a p=\left\{ \begin{array}{cl} q&\t{if } p=aq,\\0&\t{else.} \end{array} \right.
\]
and extend by linearity.  It is easy to see that the above is a complex which is exact at $e_jA$.  Moreover, it is also exact at  ${\oplus_{t(a)=j} e_{h(a)}A}$.  To see this denote by $I$ the ideal of relations, and note first that we may write $I=\sum R_i \C{}Q+Q_1I$. Now if $(f_a)\in\t{ker}d_1$, then $\sum_{a:t(a)=j} af_a\in I$, and so we may find $g_i$ such that $\sum_{a:t(a)=j} af_a-\sum_{R_i\in\c{R}}R_ig_i\in Q_1I$.  For any $a\in Q_1$ such that $t(a)=j$, we apply $\partial_a$ to this expression to obtain $f_a\equiv \sum_{R_i\in\c{R}}\partial_aR_ig_i$ mod $I$.  Thus $(f_a)$ is the image of $(g_i)$ under $d_2$, as required.

\begin{lemma}\label{projres_alphat=2}
If $\alpha_t=2$ for some $1\leq t\leq n$, then the simple $D_t$ at
vertex $t$ has projective resolution
\[
\xymatrix{0\ar[r] & e_tA\ar[r]& e_{t-1}A\oplus e_{t+1}A\ar[r] &
e_tA\ar[r]& D_t\ar[r]&0}
\]
where if $t=n$ take $t+1=0$.  Hence $\t{pd}(D_t)=2$.
\end{lemma}
\begin{proof}
We just need to show that $d_2$ is injective.  But here $d_2$ is the map
\[
\begin{array}{rcl}
e_tA&\to &e_{t-1}A\oplus e_{t+1}A \\
g&\mapsto&(\an{t-1}{t}g,-\cl{t+1}{t}g)
\end{array}
\]
and if $g\in\t{ker}d_2$, then $\an{t-1}{t}g=0$ from which (viewing as polynomials via Theorem~\ref{endo specials iso reconstruct}) we deduce that $g=0$.  Since the first syzygy is not projective (otherwise on localising it would contradict the depth lemma), it follows that $\t{pd}(D_t)=2$.
\end{proof}
\begin{cor}\label{projres_of_0_inSL}
If $\alpha_1=\hdots=\alpha_n=2$, then the simple $D_0$ at vertex
$0$ has projective resolution
\[
\xymatrix{0\ar[r] & e_0A\ar[r]& e_nA\oplus e_1A\ar[r] &
e_0A\ar[r]& D_0\ar[r]&0}
\]
and so $\t{pd}(D_0)=2$.
\end{cor}
\begin{proof}
By hypothesis the quiver is symmetric, and so the $0^{th}$
vertex is indistinguishable from the other vertices.  The result
now follows from Lemma~\ref{projres_alphat=2} above.
\end{proof}

\begin{lemma}\label{projres_alphatBIG}
If $\alpha_t>2$, then the simple $D_t$ at vertex $t$ has projective
resolution
\[
\xymatrix@R=15pt@C=15pt{0\ar[r] & (e_tA)^{\alpha_t-1}\ar[r]&
(e_{t-1}A)\oplus(e_0A)^{\alpha_t-2}\oplus(e_{t+1}A)\ar[r] &
e_tA\ar[r]& D_t\ar[r]&0}
\]
and so $\t{pd}(D_t)=2$.
\end{lemma}
\begin{proof}
Again we just need to show that $d_2$ is injective.  Here $d_2$ is the map sending $(g_i)\in(e_tA)^{\alpha_t-1}$ to 
\begin{multline*}
(\an{t-1}{t},-\CL{0}{t},0,\hdots,0)g_1+\sum_{i=1}^{\alpha_t-3}(\underbrace{0,\hdots,0}_{i},\AN{0}{t},-\CL{0}{t},\underbrace{0,\hdots,0}_{\alpha_t-i-2})g_{i+1}\\+(0,\hdots,0,\AN{0}{t},-\cl{t+1}{t})g_{\alpha_t-1},
\end{multline*}
where the convention is that the sum is empty if $\alpha_t=3$.  Now if $(g_i)\in\t{ker}d_2$, then by inspecting the first summand we see that $\an{t-1}{t}g_1=0$, and so $g_1=0$.  Now by inspecting the second summand (and using the fact that $g_1=0$), we see that $\AN{0}{t}g_2=0$, and so $g_2=0$.  Proceeding inductively gives $g_1=\hdots=g_{\alpha_t-1}=0$, and so $d_2$ is injective, as required.
\end{proof}

\begin{lemma}\label{projres_of_0_inGL}
If some $\alpha_t>2$, then the simple $D_0$ at vertex 0 has
projective resolution
\[
\xymatrix@R=15pt@C=15pt{0\ar[r] &
{\bigoplus_{i=1}^{n}(e_iA)^{\alpha_i-2}}\ar[r]&
(e_0A)^{1+\sum(\alpha_t-2)}\ar[r]& e_{n}A\oplus e_{1}A\ar[r] &
e_0A\ar[r]& D_0\ar[r]&0}
\]
and so $\t{pd}(D_0)=3$.
\end{lemma}
\begin{proof}
Denote $\gamma=\sum_{t=1}^{n}(\alpha_t-2)$.  By assumption
$\gamma\geq 1$.  Here $d_2$ is the map 
\[
\begin{array}{ccl}
(e_0A)^{1+\gamma}&\rightarrow &e_{n}A\oplus e_{1}A \\
(g_i) & \mapsto &\sum_{t=1}^{1+\gamma}(\CLH{n}{l_t}k_t,-\ANH{1}{l_{t-1}}k_{t-1})g_i
\end{array}
\]
where 
\[
\begin{array}{ccc}
\CLH{i}{j}=\left\{\begin{array}{ll} \CL{i}{j} & i\neq j \\ e_i &
i=j \end{array}\right. & \t{and} &
\ANH{i}{j}=\left\{\begin{array}{ll} \AN{i}{j} & i\neq j \\ e_i &
i=j \end{array}\right.
\end{array}
\]
and recall $k_0=\cl{1}{0}$ and $k_{\gamma+1}=\an{n}{0}$.   We first claim that the kernel of $d_2$ is 
\[
K_3:=\sum_{i=1}^{\gamma}(\underbrace{0,\hdots,0}_{i-1},\AN{0}{l_i},-\CL{0}{l_i},\underbrace{0,\hdots,0}_{\gamma-i})e_{l_i}A.
\]
To prove this claim we proceed by induction. Take
$(h_i)=(h_1,\hdots,h_{\gamma+1})$ belonging to the kernel of $d_2$.  If
$h_1=\hdots=h_{\gamma-1}=0$, then
\[
\CLH{n}{l_\gamma}k_{\gamma}h_{\gamma}=-\CLH{n}{l_{\gamma+1}}k_{\gamma+1}h_{\gamma+1}=-\an{n}{0}h_{\gamma+1},
\]
and so viewing everything as polynomials in the web of paths we
have
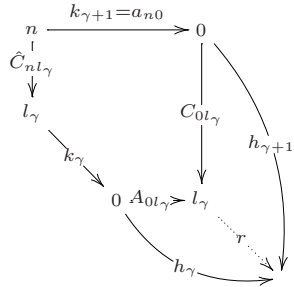
\[
\xymatrix@R=20pt@C=20pt{{}_{n}\ar[d]|(0.4){\CLH{n}{l_\gamma}}\ar[0,2]^{k_{\gamma+1}=\an{n}{0}}&
& {}_{0}\ar[2,0]|{\CL{0}{l_\gamma}}\ar@/^1pc/[3,1]|{h_{\gamma+1}}&\\ {}_{l_\gamma}\ar[dr]|{k_{\gamma}}&&&\\
&{}_{0}\ar[r]|(0.4){\AN{0}{l_\gamma}}\ar@/_1pc/[1,2]|{h_{\gamma}} & {}_{l_\gamma}\ar@{.>}[dr]|{r} &\\
& & &}
\]
We deduce that $h_\gamma=\AN{0}{l_\gamma}r$ and
$h_{\gamma+1}=-\CL{0}{l_\gamma}r$ for some $r\in e_{l_\gamma}A$ by
viewing everything as polynomials and examining the $x$ and $y$ components.  Thus
\[
(h_i)=(0,\hdots,0,h_{\gamma},h_{\gamma+1})=(0,\hdots,0,\AN{0}{l_{\gamma}},-\CL{0}{l_\gamma})r\in
K_3,
\]
and so the claim is true when $h_1=\hdots=h_{\gamma-1}=0$.  Thus
assume that the claim is true for any
$(\underbrace{0,\hdots,0}_i,h_{i+1},\hdots,h_{\gamma+1})$
belonging to the kernel with $1\leq i\leq \gamma-1$.  We shall now
show that the claim is true for any
$(\underbrace{0,\hdots,0}_{i-1},h_{i},\hdots,h_{\gamma+1})$
belonging to the kernel.  Certainly
\begin{eqnarray}
\CLH{n}{l_i}k_ih_i=-\sum_{t=i+1}^{\gamma+1}\CLH{n}{l_t}k_th_t,\label{2}
\end{eqnarray}
and so by viewing everything as polynomials and examining the $y$ components we see that
\[
y\,\,\t{component of }h_{i} \geq (y \,\,\t{component of }k_{i+1})-(y\,\, \t{component of }k_{i})=j_{l_i}.
\]
Thus
\[
\xymatrix@C=30pt{{}_{0}\ar[1,2]|{h_i}\ar[r]|{\AN{0}{l_i}} & {}_{l_i}\ar@{.>}[dr] & \\
& &}
\]
and so by viewing everything as polynomials we get that $h_i=\AN{0}{l_i}r$ for some $r\in e_{l_i}A$. Hence
\[
\CLH{n}{l_i}k_ih_i=\CLH{n}{l_i}k_i\AN{0}{l_i}r=\CLH{n}{l_{i+1}}k_{i+1}\CL{0}{l_i}r,
\]
and so (\ref{2}) becomes
\[
\CLH{n}{l_{i+1}}k_{i+1}(\CL{0}{l_{i}}r+h_{i+1})+\sum_{t=i+2}^{\gamma+1}\CLH{n}{l_t}k_th_t=0.
\]
But also
\[
-\ANH{1}{l_{i-1}}k_{i-1}h_i=-\ANH{1}{l_{i-1}}k_{i-1}\AN{0}{l_i}r=-\ANH{1}{l_i}k_i\CL{0}{l_i}r,
\]
and so by the inductive hypothesis
\[
(\underbrace{0,\hdots,0}_i,\CL{0}{l_{i}}r+h_{i+1},h_{i+2},\hdots,h_{\gamma+1})\in
K_3.
\]
But now
\begin{multline*}
(\underbrace{0,\hdots,0}_{i-1},h_{i},\hdots,h_{\gamma+1})=(\underbrace{0,\hdots,0}_{i-1},\AN{0}{l_i},-\CL{0}{l_i},0,\hdots,0)r\\+(\underbrace{0,\hdots,0}_i,\CL{0}{l_{i}}r+h_{i+1},h_{i+2},\hdots,h_{\gamma+1}),
\end{multline*}
and so
$(\underbrace{0,\hdots,0}_{i-1},h_{i},\hdots,h_{\gamma+1})\in
K_3$.  Thus by induction the claim is established, so the kernel
is $K_3$.  

We have an obvious surjection
\[
{\bigoplus_{i=1}^{\gamma}e_{l_i}A}={\bigoplus_{i=1}^{n}(e_{i}A)^{\alpha_i-2}}\rightarrow
K_3=\sum_{i=1}^{\gamma}(\underbrace{0,\hdots,0}_{i-1},\AN{0}{l_i},-\CL{0}{l_i},\underbrace{0,\hdots,0}_{\gamma-i})e_{l_i}A,
\]
which, by using a similar argument as in Lemma~\ref{projres_alphatBIG}, is also injective. This proves that $\t{pd}(D_0)\leq 3$.  Since the first and second syzygies are not projective, equality holds.
\end{proof}

Summarizing what we have proved
\begin{thm}\label{pd_of_simplesD_0...D_n}
Consider $G=\frac{1}{r}(1,a)$ and $A=A_{r,a}$.  Then for $1\leq
t\leq n$ the simple $D_t$ at vertex $t$ has projective resolution
\[
\xymatrix@R=15pt@C=15pt{0\ar[r] & (e_tA)^{\alpha_t-1}\ar[r]&
(e_{t-1}A)\oplus(e_0A)^{\alpha_t-2}\oplus(e_{t+1}A)\ar[r] &
e_tA\ar[r]& D_t\ar[r]&0}
\]
(where if $t=n$ take $t+1=0$), and so $\t{pd}(D_t)=2$.  Further, the
simple $D_0$ at vertex $0$ has projective resolution
\[
\xymatrix@R=15pt@C=15pt{0\ar[r] &
{\bigoplus_{i=1}^{n}(e_iA)^{\alpha_i-2}}\ar[r]&
(e_0A)^{1+\sum(\alpha_t-2)}\ar[r]& e_{n}A\oplus e_{1}A\ar[r] &
e_0A\ar[r]& D_0\ar[r]&0}
\]
and so\\
\t{(i)} If $G\leq SL(2,\C{})$ (i.e. all $\alpha_t=2$), then
$\t{pd}(D_0)=2$.\\
\t{(ii)} If $G\nleq SL(2,\C{})$ (i.e. some $\alpha_t>2$), then
$\t{pd}(D_0)=3$.
\end{thm}
\begin{proof}
For $1\leq t\leq n$ if $\alpha_t=2$, use
Lemma~\ref{projres_alphat=2}; if $\alpha_t>2$, then use
Lemma~\ref{projres_alphatBIG}.  For the $0^{th}$ vertex, use either
Corollary~\ref{projres_of_0_inSL} or
Lemma~\ref{projres_of_0_inGL}.
\end{proof}

All the hard work in the global dimension statement has now been
done.  To finish the proof we use standard ring theoretic methods
involving the Azumaya locus.

\begin{thm}
\[
\t{gldim}A_{r,a}=\left\{\begin{array}{l} 2 \mbox{ if }a=r-1,\\
3 \mbox{ else. }\end{array}\right.
\]
\end{thm}
\begin{proof}
It is well known by \cite{Rainwater} that
\[
\t{gldim}A=\t{sup}\{ \t{pd}_A S : S\,\, \t{a simple right } R \t{ module}
\} .
\]
Let $S$ be such a simple, and consider $\t{ann}_RS$; it is a maximal ideal of $R$ (see 
e.g. \cite[III.1.1(3)]{Brown_Goodearl}).
There are two possibilities:\\
(i) $\t{ann}_RS$ lies in the Azumaya locus.  Then
\[
\t{pd}_AS=\t{sup}\{\t{pd}_{A_\m} S_\m : \m\in \t{max}R
\}=\t{pd}_{A_{\t{ann}_R
S}}S_{\t{ann}_RS}\leq\t{gldim}A_{\t{ann}_RS}=2,
\]
by Corollary~\ref{gldim_az}.\\
(ii) $\t{ann}_RS$ does not lie in the Azumaya locus, so by Lemma~\ref{az} $\t{ann}_RS=\m_0$.  Now the maximal number of non-isomorphic simple $A$-modules $V$ with $\t{ann}_RV=\m_0$ is equal to the PI degree of $A$ (\cite[III.1.1(3)]{Brown_Goodearl}), which we already know is $n+1$.  But it is clear that $D_0,\hdots,D_n$ are all annihilated by $\m_0$, and so consequently these must be all the simple $A$-modules annihilated by $\m_0$.  Thus $S$ must be one of $D_0,\hdots,D_n$, and so by Theorem~\ref{pd_of_simplesD_0...D_n} we know that the projective dimension is either 2 or 3.\\
Combining (i) and (ii) gives the desired result.
\end{proof}

\end{document}